\documentclass[12pt,reqno]{amsart}
\usepackage[dvipsnames]{xcolor}
\usepackage[colorlinks=true,linkcolor=Blue,citecolor=NavyBlue]{hyperref}
\usepackage[T1]{fontenc}
\usepackage{geometry}
\usepackage{amsmath,amssymb}
\usepackage{amsfonts}
\usepackage{eucal}
\usepackage{amsthm}
\usepackage{etoolbox}
\usepackage{mathrsfs}
\numberwithin{equation}{section}
\usepackage{comment}
\usepackage{tcolorbox}

\usepackage{mathtools}\usepackage{tikz}
\usetikzlibrary{arrows.meta,decorations.markings}

\makeatletter
\def\author@andify{%
  \nxandlist
    {\unskip, \penalty-1\space\ignorespaces}%
    {\unskip, \penalty-2\space\ignorespaces}%
    {\unskip, \penalty-2\space\ignorespaces}%
}
\makeatother

\def\a{\alpha}
\def\b{\beta}
\def\c{\gamma}

\def\f{\varphi}
\def\g{\psi}

\let\ogonekaccent\k
\renewcommand\k{\kappa}

\def\s{\sigma}

\def\x{\xi}
\def\y{\eta}
\def\z{\zeta}

\newcommand{\La}{\Lambda}

\def\re{\mathbb{R}}

\def\pa{\partial}

\renewcommand{\Re}{\text{{\rm Re}\;}}

\newcommand{\supp}{\text{{\rm supp}\;}}

\newcommand{\Ker}{\text{{\rm Ker}\;}}

\newcommand{\vol}{\text{\rm vol}}

\newcommand{\sgn}{\text{{\rm sgn}\;}}
\newcommand{\proofsubtitle}[1]{\par\smallskip\noindent\textit{#1}\ }
\newcounter{proofstep}
\AtBeginEnvironment{proof}{\setcounter{proofstep}{0}}
\newcommand{\proofstep}[1][]{%
  \refstepcounter{proofstep}%
  \proofsubtitle{Step \theproofstep\ifstrempty{#1}{}{ (#1)}.}%
}

\newtheorem{thm}{Theorem}[section]
\newtheorem{lem}[thm]{Lemma}
\newtheorem{prop}[thm]{Proposition}
\newtheorem{cor}[thm]{Corollary}

\theoremstyle{definition}
\newtheorem{defn}[thm]{Definition}

\theoremstyle{remark}
\newtheorem{rem}[thm]{Remark}

\newcommand{\WF}{\mathrm{WF}}

\title[Large eigenvalues of the Connes--Moscovici operator]{Large eigenvalues of the Connes--Moscovici operator}

\author{Kouichi Taira}
\address{Faculty of Mathematics, Kyushu University, Fukuoka, Japan}
\email{taira.kouichi.800@m.kyushu-u.ac.jp}

\author{Max Willems}
\address{Mathematical Institute, Universiteit Utrecht, Utrecht, The Netherlands}
\email{m.j.a.willems@students.uu.nl}

\author{Micha{\l} Wrochna}
\address{Mathematical Institute, Universiteit Utrecht, Utrecht, The Netherlands \vspace{-0.3cm}}
\address{Mathematics \& Data Science, Vrije Universiteit Brussel, Brussels, Belgium}
\email{m.wrochna@uu.nl} 
\email{michal.wrochna@vub.be}

\makeatletter \def\ps@firstpage{
 \ps@plain \def\@oddfoot{\normalfont\scriptsize \hfil\raisebox{-\baselineskip}[0pt][0pt]{\thepage}\hfil} \let\@evenfoot\@oddfoot } 
 \makeatother

\begin{document}

\begin{abstract}
We prove the logarithmic Weyl law predicted by Connes and Moscovici for large positive eigenvalues of the distinguished self-adjoint extension of the operator $P=-\partial_x(x^2-1)\partial_x-4\pi^2x^2$ on the real line. More precisely,  we use the WKB method and show a Bohr--Sommerfeld quantization rule  characterizing sufficiently large eigenvalues. A particular difficulty is that the associated $h$-dependent Lagrangian
submanifolds are non-compact, which requires a renormalization of the action integral.  Furthermore, to deal with the degeneracy of the coefficients at $x =\pm 1$, we use a quasimode matching argument based on radial estimates and Lagrangian regularity.
\end{abstract}

\maketitle

\newcommand\blfootnote[1]{%
  \begingroup
  \renewcommand\thefootnote{}\footnote{#1}%
  \addtocounter{footnote}{-1}%
  \endgroup
}


\section{Introduction}

\subsection{Introduction and main result}\label{subsection:model-main-result}

Weyl laws relate the leading term in the asymptotics of the eigenvalue counting function of a differential operator to the volume of an appropriate region in phase space. When the highest-order coefficients degenerate, the volume can be infinite, or the leading  contribution may concentrate near a singular set, in which case the resulting asymptotics may involve non-classical powers or logarithmic factors. Degenerate elliptic and hypoelliptic operators, singular metrics, and related non-classical regimes have been studied in many settings, see \S\ref{subsection:logarithmic-weyl-laws}. However, an all-encompassing theory is not presently available.

\medskip

In this paper we study the differential expression
\begin{align}\label{eq:Psa-CM}
P=-\partial_x(x^2-1)\partial_x-4\pi^2x^2
\end{align}
on the \textit{full} real line. It is a particularly simple model in which the leading coefficient has two simple zeros and changes sign. From the point of view of Sturm--Liouville theory, the real line is naturally divided into the three intervals $(-\infty,-1)$, $(-1,1)$, and $(1,\infty)$. Near $x=\pm 1$, the associated Sturm--Liouville equation admits one bounded solution and one logarithmic solution, so these endpoints are of limit-circle type. At infinity, the two independent solutions have leading form $|x|^{-1}e^{\pm2\pi ix}$ and are both square-integrable. Thus, specifying a self-adjoint realization of $P$ in $L^2(\re)$ requires boundary conditions at both $|x|=1$ and $|{x}|=\infty$.

The operator $-P$ restricted to the middle interval $(-1,1)$ is the well-studied \emph{prolate spheroidal operator} and has a pretty straightforward spectral theory. Namely, the boundary condition
\begin{align}\label{eq:prolate-no-log}
\lim_{x\to\pm1}(1-x^2)\partial_xu(x)=0
\end{align}
eliminates the logarithmic term and defines a self-adjoint operator with compact resolvent and quadratically growing eigenvalues. The eigenfunctions are the so-called prolate spheroidal wave functions, which frequently arise in applications as they diagonalize the truncated Fourier transform \cite{SP,ORX}.

As observed by Connes and Moscovici \cite{CM}, considering $P$ on the whole real line gives rise to new spectral phenomena. While there are many self-adjoint extensions, there is a distinguished one, denoted here by $P_{\rm sa}$, which enjoys special commutation properties with the truncated Fourier transform. It can be defined by imposing condition \eqref{eq:prolate-no-log} at $x=\pm1$ and  selecting a sine (resp.~cosine) type asymptotic for the even (resp.~odd) part
of functions in the maximal domain  (see \S\ref{subsection:boundary-data}). The self-adjoint operator $P_{\rm sa}$ is shown in \cite{CM} to have discrete spectrum.

\medskip

A special role in the analysis is played by the function
\begin{align}\label{eq:Iadef}
I(a)=\int_1^{\infty}\biggl(\sqrt{\frac{x^2+a-1}{x^2-1}}-1\biggr)dx,
\quad a>0.
\end{align}
This function arises as the renormalized action associated with one component of the zero set $\{p_h=0\}$ analysed in \S\ref{subsection:semiclassical-dynamics}. Our main result is the following theorem, which characterizes all sufficiently large eigenvalues of $P_{\rm sa}$ by a Bohr--Sommerfeld quantization rule.

\begin{thm}\label{thm:BSthm}
There exist ${E_*}>0$ and $R\in C^\infty([{E_*},\infty))$ such that
\begin{align*}
R(E)=-\frac{1}{16\pi}E^{-\frac12}\log E+O(E^{-\frac12}),\quad R'(E)=O(E^{-\frac32}\log E)
\end{align*}
as $E\to\infty$, and the following holds true. The positive spectrum of $P_{\rm sa}$ above ${{E_*}}$ consists of simple eigenvalues $E_n$, $n\in\frac12\mathbb Z$, $n\gg1$, satisfying
\begin{align}\label{BS}
2I\biggl(1+\frac{E_n}{4\pi^2}\biggr)+\frac14+R(E_n)
=n+O(E_n^{-\infty}).
\end{align}
\end{thm}
As a corollary, we obtain the following logarithmic Weyl law for the positive spectrum of $P_{\rm sa}$; see \S\ref{ss:finalproof} for the precise argument.

\begin{cor}\label{cor:Weyllaw}
Let $N_+(E)$ denote the number of positive eigenvalues of $P_{\rm sa}$ in $(0,E]$, counted with multiplicity. Then
\begin{equation}\label{eq:weyl}
N_+(E)=4\frac{\sqrt E}{2\pi}
\biggl(\log\frac{\sqrt E}{2\pi}-1+2\log2\biggr)+O(1),
\quad E\to\infty.
\end{equation}
\end{cor}

The non-classical Weyl law \eqref{eq:weyl} was conjectured by Connes and Moscovici, who motivated it in \cite[\S5]{CM} by a heuristic phase space volume argument. We remark that Theorem \ref{thm:BSthm} gives more precise information than the Weyl law itself. To prove these results one cannot use classical results in Sturm--Liouville theory such as  \cite[(1.4)]{Heywood}, \cite[(1.4)]{NR}, \cite[Thm.~1]{AF}, as the hypotheses are not satisfied by the corresponding reduced Schr\"odinger operator at the relevant endpoints.

\begin{rem}\label{rem:Maslov-quarter}
The term $\frac14$ in the Bohr--Sommerfeld condition \eqref{BS} comes from the Maslov (Morse) index in the stationary phase computation in \eqref{eq:stationary-phase-match}. Neither this term nor the remainder $R(E)$ contributes to the leading Weyl asymptotics.
\end{rem}

\begin{rem}\label{rem:general-mu}
The value $4\pi^2$ is distinguished by Fourier transform commutation properties, see \cite{CM}, but it is not essential for our arguments. If $4\pi^2x^2$ is replaced by $\mu x^2$, $\mu>0$, and the conditions  at infinity are modified by replacing $2\pi$ with $\sqrt\mu$, the same proof applies.  The quantization rule becomes
\begin{align*}
\frac{\sqrt\mu}{\pi}\left[ I\left(1+\frac{E_n}{\mu} \right)-1 \right]+\frac14+R_\mu(E_n)
=n+O(E_n^{-\infty}),
\end{align*}
where $R_\mu(E)=-\frac{1}{16\pi}E^{-1/2}\log E+O(E^{-1/2})$ (with a $\mu$-dependent $O(E^{-1/2})$ remainder). Consequently, a repetition of the arguments in \S\ref{ss:finalproof} gives the more general Weyl law
\begin{align*}
N_{+,\mu}(E)=\frac{2\sqrt E}{\pi}
\biggl(\log\frac{\sqrt E}{\sqrt\mu}-1+2\log2\biggr)+O(1).
\end{align*}
\end{rem}

\begin{rem}\label{rem:cpo}
The $O(1)$ remainder in the Weyl law \eqref{eq:weyl} cannot be replaced by $c+o(1)$ for any constant $c$. Indeed, suppose that $N_+(E)=F(E)+c+o(1)$, where
\[
F(E)=4\frac{\sqrt{E}}{2\pi}\left(\log\left(\frac{\sqrt{E}}{2\pi}\right)-1+2\log 2\right).
\]
Let $E\gg1$ be an eigenvalue of $P_{\rm sa}$. Since the spectrum is discrete and the eigenvalues are simple, we may choose $0<\varepsilon_E\leq1$ such that $E$ is the only eigenvalue in $(E-\varepsilon_E,E+\varepsilon_E)$. Thus $N_+(E+\varepsilon_E)-N_+(E-\varepsilon_E)=1$. On the other hand, the assumed asymptotics and the mean value theorem give
\[
N_+(E+\varepsilon_E)-N_+(E-\varepsilon_E)
=F(E+\varepsilon_E)-F(E-\varepsilon_E)+o(1)
=O\left(\varepsilon_E\frac{\log E}{\sqrt E}\right)+o(1)
=o(1),
\]
as $E\to +\infty$, which is a contradiction.
\end{rem}

\subsection{Structure of the proof}\label{subsection:proof-structure}
Our approach has the advantage that it can handle the original problem directly and clarifies how microlocal analysis techniques can be applied in this singular setting. 
Our original motivation comes in fact from the relation between the Connes-Moscovici operator and Vasy's method \cite{VasyKerrdS, ZwVa}. 
Our proof is robust enough to be extended to more general, higher order operators (with the exception of a Wronskian argument, which can be replaced by a more robust, albeit complicated, technique) 
and different boundary conditions.

\medskip

We first introduce a small semiclassical parameter $h>0$ by setting $E=h^{-2}$ and define $P(h):=h^2(P-h^{-2})$. The zero set $\{p_h=0\}$ of the $h$-dependent symbol introduced in \S\ref{subsection:semiclassical-dynamics} has four components $\Lambda_{\sigma_1,\sigma_2}$ labeled by $\sigma_1,\sigma_2\in \{+,-\}$, extending from spatial infinity
\[
\sigma_1x\to\infty,\quad \xi\to2\pi h\sigma_2,
\]
where solutions have oscillatory asymptotics, to the degeneracy at $x=\sigma_1$, where $\sigma_2\xi\to+\infty$ and the Hamilton flow becomes radial.  We stress that although each $\Lambda_{\s_1,\s_2}$ is Lagrangian for fixed $h$, the
$h$-dependent family is neither a fixed Lagrangian of the standard
semiclassical calculus nor does the family extend smoothly as a semiclassical scattering
Lagrangian at spatial infinity; see
\S\ref{subsection:semiclassical-dynamics}.

\medskip

The argument is organized as follows.

\medskip

 First, on each component $\Lambda_{\sigma_1,\sigma_2}$, we construct a ``model quasimode'' (meaning a solution modulo $O(h^\infty)$ that is not necessarily in the domain of $P_{\rm sa}$). More precisely, near spatial infinity we use to that end a WKB form, but near the degeneracy we need instead a local Lagrangian distribution representation better adapted to the radial behavior. We then use a stationary phase argument on an intermediate interval to
 connect the two representations and obtain a connection formula involving
 a phase shift and a relative $O(h\log(1/h))$ correction.

   We remark that due to the non-compactness of $\Lambda_{\sigma_1,\sigma_2}$, standard WKB estimates on fixed compact sets do \textit{not} provide uniform control  along the whole component. We carry out the WKB transport estimates in a special two-scale symbol class with logarithmic weights, and then  comparison with the Lagrangian distribution  leads  us to introduce a renormalized action which replaces the formal divergent integral ``$\int_{\Lambda_{\sigma_1,\sigma_2}}\xi\,dx$''. The borderline $|x|^{-1}$ decay of the transport terms on $1\ll |x|\ll h^{-1}$ produces successive powers of $\log(1/h)$ and, in particular, the $h\log(1/h)$ correction.

\medskip

 Next, the crucial step is to prove completeness of the four model quasimodes. Arguments based on the Wronskian allow us to express any exact eigenfunction as a linear combination of model quasimodes modulo a remainder that is rapidly decreasing at spatial infinity and $O_{C^\infty}(h^\infty)$ on compact subsets of $\mathbb R\setminus\{\pm1\}$. To tackle the remaining problem at $x=\pm 1$ we first use a below-threshold radial estimate, in the form of \cite[Thm.~E.54]{DZbook}, to propagate $O(h^\infty)$ smallness of remainders in the sense of $H^{1/2-0}_h$ Sobolev norms. A  module regularity argument for fixed $h$, in the spirit of Haber--Vasy \cite{HV} and Dyatlov--Zworski \cite{DZ}, followed by a use of transport equations near fiber infinity and a bootstrap argument for the behaviour in $h$, then improves this to $O(h^\infty)$ in a stronger sense, namely, at the level of boundary data.

\medskip

 Quasimodes satisfying the self-adjoint boundary conditions are then constructed as linear combinations of the ``model quasimodes''. Imposing the conditions amounts to solving a finite-dimensional system, and equating the corresponding determinant to zero gives rise to the Bohr--Sommerfeld quantization condition.  Finally, standard arguments allow to go from quasimode analysis to statements on the spectrum.

\subsection{Bibliographical remarks} We comment below on various related results in the literature, either on the spectral theory of the operator $P_{\rm sa}$ or on non-classical Weyl laws.  

\subsubsection{The Connes--Moscovici operator}\label{subsection:CM-conventions}

We write $D_x=i^{-1}\partial_x$. Connes and Moscovici show that $P$ has a self-adjoint  extension $P_{\rm sa}$ uniquely determined by the property of commutation with both the multiplication operator ${\mathbf{1}}_{[-1,1]}(x)$ and the Fourier multiplier ${\mathbf{1}}_{[-2\pi,2\pi]}(D_x)$  \cite{CM}. They also prove that  $P_{\rm sa}$ has discrete spectrum, unbounded in both directions. Furthermore, it is shown in \cite{CM} that the eigenfunctions of $P_{\rm sa}$ with positive eigenvalues are $L^2$ eigenfunctions of $P$ which lie in the \emph{Sonin space} 
\begin{equation}\label{eq:sonin}
\Ker {\mathbf{1}}_{[-1,1]}(x)\cap \Ker {\mathbf{1}}_{[-2\pi,2\pi]}(D_x).
\end{equation}

\begin{rem}Note that the operator considered in \cite{CM} is $W_{\rm sa}=-P_{\rm sa}$, hence the terminology there refers instead to \emph{negative}  eigenvalues. Note also that the parameter denoted by $E$ in \cite{CM} is the square root of our eigenvalue variable.\end{rem}


\medskip

The work of Connes and Moscovici is motivated by the problem of finding a spectral realization of the non-trivial zeros of the Riemann zeta function in the spirit of the Hilbert--P\'olya conjecture. We do not address this problem here: although Corollary~\ref{cor:Weyllaw} confirms the Weyl law predicted by \cite{CM}, we stress that the comparison with the zeta zeros that one can infer at this point concerns only the leading asymptotics of their counting function, as given by the Riemann--von Mangoldt formula. Our result neither identifies individual (square roots of)  eigenvalues of $P_{\rm sa}$ with the zeta zeros nor reproduces the lower-order oscillations in the zero-counting function. A spectral realization of the zeta zeros, as sought in the broader programme of Connes and collaborators \cite{CCM, CM}, would also have to account for these oscillations.

Several recent works develop complementary aspects of the analysis. The Sonin space occurs independently in the approach of Connes and Consani to Weil positivity \cite{ConnesConsaniWeil}, which has further motivated considering the prolate operator $P_{\rm sa}$ in a number-theoretical context. In a subsequent work \cite{CCM}, Connes--Consani--Moscovici introduce semilocal analogues of the prolate operator $P_{\rm sa}$ and of the Sonin space, incorporating in addition local components associated with finitely many primes. They show that the resulting generalized Sonin spaces remain isomorphic when further primes are added, and realize them as Hilbert spaces of entire functions. Ramis--Richard-Jung--Thomann develop an alternative complex-analytic description of the eigenvalues of $P_{\rm sa}$ that do not arise from the restricted problem $P|_{[-1,1]}$ \cite{RRT}. More precisely, they construct spectral determinants, characterize these eigenvalues through decay properties of solutions analytically continued to the complex plane, prove that they are all positive (thus, they are \emph{negative} eigenvalues in the terminology for $W_{\rm sa}=-P_{\rm sa}$), and give a fast method for their numerical computation.  Our Bohr--Sommerfeld formula complements this characterization by
providing high-energy asymptotics (which prove in particular that the bound on determinant orders in \cite[Prop.~16]{RRT} is sharp).

\subsubsection{Degenerate operators and logarithmic Weyl laws}\label{subsection:logarithmic-weyl-laws}

In this  subsection,  $N(E)$ denotes the number of eigenvalues, counted with multiplicity, in $[0,E]$ for the self-adjoint operator specified in the surrounding sentence. We write $N_+(E)$ to mean the number of positive eigenvalues in $(0,E]$ when we want to stress that the operator is not bounded below.

Recall that for a positive elliptic differential operator $A$ of order $m$ on a closed $d$-dimensional manifold $X$, the classical Weyl law reads
\begin{align*}
N(E)\sim(2\pi)^{-d}\vol\{(x,\xi)\in T^*X\mid a_m(x,\xi)\leq E\},
\end{align*}
where $a_m$ is the principal symbol of $A$. The homogeneity of $a_m$ in $\xi$ is  then responsible for  the familiar power law $E^{d/m}$. If the highest-order coefficients degenerate, however, the formal phase-space volume may  diverge, the leading contribution may concentrate near a singular stratum, and additional powers or logarithms may occur. 

There is by now a broad, but necessarily heterogeneous, theory. Strong power degeneracies at a boundary were studied by Vulis--Solomyak \cite{VulisSolomyak}; multiple characteristics and hypoelliptic operators were studied by M\'etivier and Menikoff--Sj\"ostrand \cite{Metivier,MenikoffSjostrand}, with a wider microlocal perspective given by Ivrii in \cite{Ivrii}. Non-compact classically allowed regions lead to the non-classical laws of Simon and Aramaki--Nurmuhammad \cite{Si,AN}, see below. More recent work treats sub-Riemannian and non-equiregular geometries \cite{CHT}, almost-Riemannian Laplacians \cite{BPS}, singular Riemannian manifolds \cite{CPR}, singular metric-measure spaces \cite{DHPW}, and boundary degeneracies with a variable exponent \cite{CDT}.

A recent work by Colin de Verdi\`ere studies sign-changing Laplace operators defined by an elliptic transmission problem across a smooth hypersurface \cite{CdVsign}. The author proves separate classical Weyl laws for the positive and negative spectra and constructs modes concentrated near the interface. The sign-changing coefficient is a feature shared with the prolate operator $P$, although both the geometric setting and the spectral mechanism are actually quite different.

Seemingly closer to the one-dimensional setting considered here, Hartmann--Lesch-\-Vertman \cite{HLV} study Sturm--Liouville operators on a half-line of the form $-\partial_x(x^2\partial_x)+\mu^2x^2-\frac14+V(x)$ and, among other results, prove the logarithmic Weyl law $N(E)\sim (2\pi)^{-1}\sqrt E\log E$. The sign of the quadratic potential is however \textit{opposite} to that of $P$, and $x^2$ has \textit{no zero} on the interval under consideration. After a logarithmic change of variables the model reduces to $-\partial_t^2+\mu^2e^{2t}$, so unlike in our problem, infinity is confining and limit point with a distinguished decaying solution.

Logarithmic factors are especially characteristic of borderline regimes, but they arise from different mechanisms. For $\alpha>0$, Simon  \cite{Si}  considers the Schr\"odinger operator $-\Delta+|xy|^\alpha$ on $L^2(\mathbb R^2)$ and proves that its counting function satisfies
\begin{align*}
N(E)=\frac1\pi E^{\frac{\alpha+1}{\alpha}}\log E
+o\bigl(E^{\frac{\alpha+1}{\alpha}}\log E\bigr),
\end{align*}
 For the Dirichlet Laplacian on the infinite-volume horn $\Omega=\{(x,y)\in\mathbb R^2\mid |xy|\leq1\}$, Aramaki--Nurmuhammad show that the counting function satisfies
\begin{align*}
N(E)=\frac1\pi E\log E+o(E\log E),
\end{align*}
with related higher-dimensional results \cite{AN}. Fox--Strichartz consider the wave operator $\partial_x^2-\partial_y^2$ on the flat two-torus $\mathbb{T}^2$ and show 
\begin{align*}
N_+(E)=E\log E+(2\gamma-1)E+O(E^{1/2})
\end{align*}
by an arithmetic lattice-counting argument \cite[Thm.~3.2]{FS}. For the Laplacian on a compact two-dimensional singular metric-measure space with a synthetic lower Ricci-curvature bound, constructed by Dai--Honda--Pan--Wei, one has
\begin{align*}
N(E)=\frac1{4\pi}E\log E+o(E\log E)
\end{align*}
\cite{DHPW}. In the problem with variable boundary degeneracy considered by Colin de Verdi\`ere--Dietze--Tr\'elat \cite{CDT}, the leading regime is determined by whether the maximal degeneracy exponent lies below, at, or above a critical value; under a Morse--Bott hypothesis on its maximum set, the asymptotics acquire logarithmic factors whose form depends on the codimension of that set.

\medskip

From the microlocal point of view, the behavior near $x=\pm1$
is closest to Vasy's analysis of radial sets conormal to a hypersurface in asymptotically hyperbolic and Kerr--de Sitter problems
\cite{VasyKerrdS,VasyPoisson}; see also 
\cite{ZwVa}. The module regularity argument that we use is a simple version of  Haber and Vasy's \cite{HV}, while the passage from radial estimates to a conormal oscillatory representation and  the subsequent analysis of the transport equations follows the pattern of Dyatlov--Zworski \cite{DZ}. The $h$-dependent symbol classes that we use do not seem to have entered the literature so far, we remark however that there are similarities with Hintz's semiclassical cone calculus and its second microlocal refinement
\cite{HintzCone,HintzConePropagation}. This suggests an alternative approach based on blowing up a corner of compactified
semiclassical phase space, which we do not pursue here.

\subsection{Organization of the paper}

In \S\ref{section:geometry} we introduce preliminaries on the semiclassical analysis of $P(h)$ and boundary conditions. In \S\ref{section:models} we construct the four model quasimodes and compute the connection formula. In \S\ref{section:completeness} we prove the completeness of model quasimodes using  arguments based on the Wronskian, radial estimates, conormal regularity, and transport equations at fiber infinity. Then in \S\ref{section:Mainproof} we impose the boundary conditions, we prove spectral localization and simplicity, and we conclude the Weyl law. Appendix \ref{app:saP} reviews the description of self-adjoint extensions of $P$ on the real line.   Finally, the symbolic estimates used in the quasimode construction are collected in Appendix \ref{appendix:symbolic-estimates}.

\section{Preliminaries}\label{section:geometry}

\subsection{Semiclassical dynamics}\label{subsection:semiclassical-dynamics}

To study the eigenvalue problem for $E\to\infty$, we introduce a small semiclassical parameter $h>0$ by setting $E=h^{-2}$ and write
\begin{align}\label{eq:P(h)def}
P(h)=h^2(P-h^{-2})=-h^2\partial_x(x^2-1)\partial_x-4\pi^2h^2x^2-1.
\end{align}
We denote by
\begin{align}\label{eq:full-symbol}
p_h(x,\x)=(x^2-1)\x^2-4\pi^2h^2x^2-1
\end{align}
the real $h$-dependent symbol associated to $P(h)$ that will be used later in the global eikonal equation. On compact subsets of the $x$-line it differs by $O(h^2)$ from the semiclassical principal symbol $(x^2-1)\x^2-1$, but the additional term $-4\pi^2h^2x^2$ becomes relevant at the scale $|x|\sim h^{-1}$.

The zero set $\{p_h=0\}$ has four connected components
\begin{align}\label{eq:characteristic-components}
\La_{\s_1,\s_2}
=\{(x,\x)\mid p_h(x,\x)=0,\ \s_1x>1,\ \s_2\x>0\},
\quad \s_1,\s_2\in\{\pm\}.
\end{align}
On $\{p_h=0\}$ we have
\begin{align*}
(x^2-1)\x^2=1+4\pi^2h^2x^2,
\end{align*}
 so at spatial infinity in $\Lambda_{\sigma_1,\sigma_2}$, $\s_1x\to\infty$ and $\x\to2\pi h\s_2$, and on the other hand, when approaching the degeneracy  $x\to\s_1$ we have  $\s_2\x\to+\infty$. 
 
 \begin{rem}
 For each fixed $h>0$, $\Lambda_{\s_1,\s_2}$ is a smooth exact
 Lagrangian curve in $T^*\mathbb R$. The family is nevertheless not a
 fixed global Lagrangian for the standard semiclassical calculus, nor
 does it extend as a smooth semiclassical scattering Lagrangian family
 at spatial infinity. Indeed, with
 $r=(\s_1x)^{-1}$ and scattering fiber coordinate
 $\eta=-\s_1\x$, one has
 \[
 \eta
 =-\s_1\s_2
 \sqrt{\frac{r^2+4\pi^2h^2}{1-r^2}},
 \]
 and the leading term $\sqrt{r^2+4\pi^2h^2}$ is not smooth at the corner
 $h=r=0$. 
 \end{rem}

Following \cite[Appendix E]{DZbook}, let $\overline{T^*\mathbb R}$ be the
fiber-radial compactification of $T^*\mathbb R$, whose boundary
$\partial\overline{T^*\mathbb R}$ represents fiber infinity. The two boundary
points in each fiber are labeled by
$\widehat\x:=\x/|\x|\in\{\pm1\}$. We denote the endpoint of $\Lambda_{\s_1,\s_2}$ at fiber infinity by 
\begin{align}\label{eq:radial-set}
\mathcal L_{\s_1,\s_2}
:=\overline{\La_{\s_1,\s_2}}\cap\partial\overline{T^*\mathbb R}
=\{x=\s_1,\ \widehat\x=\s_2\},
\end{align}
where the closure is taken in $\overline{T^*\mathbb R}$. In the boundary chart
$\s_2\x>0$, the function
\[
\rho=(\s_2\x)^{-1}
\]
is a defining function of $\partial\overline{T^*\mathbb R}$, and
$\mathcal L_{\s_1,\s_2}=\{x=\s_1,\rho=0\}$.

The Hamilton vector field is
\begin{align}\label{eq:Hamilton-vector-field}
H_{p_h}
=2(x^2-1)\x\,\partial_x
 -(2x\x^2-8\pi^2h^2x)\partial_\x.
\end{align}
In the coordinates $(x,\rho)$, its rescaling by $\rho$ reads
\begin{align}\label{eq:radial-linearization}
\rho H_{p_h}
=2\s_2(x^2-1)\partial_x
 +2\s_2x\rho(1-4\pi^2h^2\rho^2)\partial_\rho.
\end{align}
Thus $\rho H_{p_h}$ extends smoothly to the fiber boundary and vanishes at
$\mathcal L_{\s_1,\s_2}$. Its linear part there is
\[
4\s_1\s_2(x-\s_1)\partial_x
 +2\s_1\s_2\rho\partial_\rho.
\]
Consequently, $\mathcal L_{\s_1,\s_2}$ is a radial source when
$\s_1\s_2=1$ and a radial sink when $\s_1\s_2=-1$, in the terminology of
\cite[Def.~E.50]{DZbook}. The direction corresponding to moving  forward along the flow is indicated by an arrow in Figure \ref{fig:phase-portrait}.

\begin{figure}[t]
  \centering
 \begin{tikzpicture}[>=Latex,scale=1.08,every node/.style={font=\small},
   axis/.style={black!70,line width=0.55pt},
   boundary/.style={black!45,line width=0.45pt},
   char/.style={line width=0.95pt},
   guide/.style={black!35,dashed,line width=0.4pt},
   flow/.style={line width=0.6pt,postaction={decorate},
     decoration={markings,mark=at position 0.58 with {\arrow{Latex[length=2.0mm,width=1.2mm]}}}},
   radial/.style={circle,fill=black,inner sep=1.3pt}]
   \def\c{0.33}\def\Y{2.25}\def\X{4.45}
   \pgfmathsetmacro{\ya}{\Y*atan(\c)/90}
   \newcommand{\prof}{\Y*(1-atan(sqrt(((\x*\x)-1)/(1+\c*\c*\x*\x)))/90)}
 
   \draw[boundary] (-\X,\Y) -- (\X,\Y); \draw[boundary] (-\X,-\Y) -- (\X,-\Y);
   \node[anchor=west,black!55] at (2.45,\Y+0.23) {$\partial\overline{T^*\mathbb R}$};
   \node[anchor=west,black!55] at (2.45,-\Y-0.23) {$\partial\overline{T^*\mathbb R}$};
 
   \draw[axis,->] (-\X,0) -- (\X+0.12,0) node[below right,yshift=1.1mm] {$x$};
   \draw[axis] (0,-\Y) -- (0,\Y);
 
   \foreach \s in {-1,1} {\draw[guide] (\s,-\Y) -- (\s,\Y); \node[below=1pt] at (\s,0) {$\s$};}
   \draw[guide] (-\X+0.15,\ya) -- (\X-0.15,\ya); \draw[guide] (-\X+0.15,-\ya) -- (\X-0.15,-\ya);
 
   \foreach \a/\b in {1/4.15,-4.15/-1}{
     \draw[char] plot[domain=\a:\b,samples=180,smooth] ({\x},{\prof});
     \draw[char] plot[domain=\a:\b,samples=180,smooth] ({\x},{-\prof});
   }
 
   \draw[flow] plot[domain=1.34:2.75,samples=80] ({\x},{\prof});
   \draw[flow] plot[domain=2.75:1.34,samples=80] ({\x},{-\prof});
   \draw[flow] plot[domain=-2.75:-1.34,samples=80] ({\x},{\prof});
   \draw[flow] plot[domain=-1.34:-2.75,samples=80] ({\x},{-\prof});
 
   \node[radial,label={[above left=-1pt]{$\mathcal L_{-+}$}}] at (-1,\Y) {};
   \node[radial,label={[below left=-1pt]{$\mathcal L_{--}$}}] at (-1,-\Y) {};
   \node[radial,label={[above right=-1pt]{$\mathcal L_{++}$}}] at (1,\Y) {};
   \node[radial,label={[below right=-1pt]{$\mathcal L_{+-}$}}] at (1,-\Y) {};
 
   \node at (2,1.3) {$\Lambda_{++}$}; \node at (2,-1.3) {$\Lambda_{+-}$};
   \node at (-2,1.3) {$\Lambda_{-+}$}; \node at (-2,-1.3) {$\Lambda_{--}$};
 
 
   \node[anchor=west,black!55] at (4.2,\ya) {$\xi=2\pi h$};
   \node[anchor=west,black!55] at (4.2,-\ya) {$\xi=-2\pi h$};
      \node[anchor=west,black!55] at (-6,\ya) {$\phantom{\xi=2\pi h}$};
      \node[anchor=west,black!55] at (-6,-\ya) {$\phantom{\xi=-2\pi h}$};
 \end{tikzpicture}
  \caption{%
    Schematic phase portrait of the zero set $\{p_h=0\}$ in the fiber-radial compactification $\overline{T^*\mathbb R}$. The four
    connected components $\Lambda_{\sigma_1,\sigma_2}$ lie over
    $|x|>1$ and approach, as $x\to\pm\infty$, the limiting
    levels $\xi=2\pi h$ or $\xi=-2\pi h$. The radial points
    $\mathcal L_{\sigma_1,\sigma_2}$ lie over $x=\pm1$ at fiber infinity
    $\partial\overline{T^*\mathbb R}$. }
  \label{fig:phase-portrait}
\end{figure}
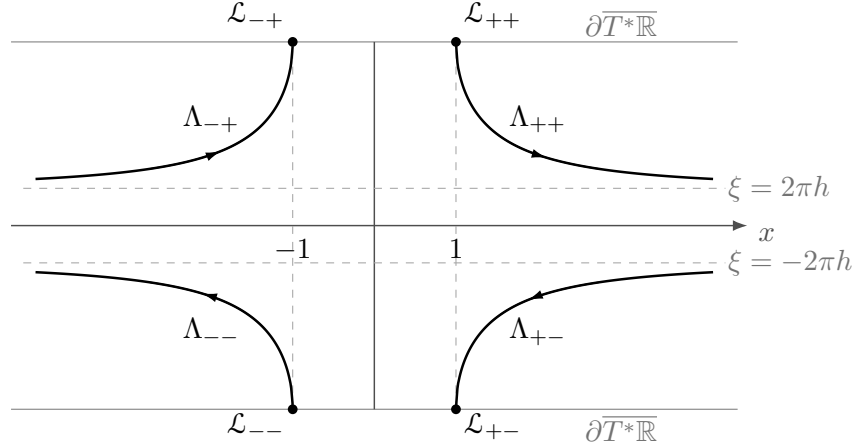

\subsection{Renormalized action}\label{subsection:renormalized-action}

In the standard one-dimensional Bohr--Sommerfeld setting, suppose that the energy surface $\{q=E\}$ has a compact connected component $\Lambda$ on which $dq\neq0$. Then $\Lambda$ is a periodic Hamiltonian orbit. On any portion of $\Lambda$ that is locally a graph $\xi=\xi(x)$, a WKB ansatz has the form $a(x;h)e^{i\varphi(x)/h}$, and the eikonal equation implies $\varphi'(x)=\xi(x)$. The phase accumulated over one period is therefore the classical action $\int_{\Lambda}\xi\,dx$. Requiring the local WKB solutions to match after continuation around $\Lambda$ gives the Bohr--Sommerfeld condition, with the Maslov index providing the first phase correction; see \cite[\S\S5--7]{VuNgocWKB} for a modern review and \cite{C1,ILR} for other approaches.

\medskip

In the present problem, the components $\Lambda_{\sigma_1,\sigma_2}$ are non-compact and the usual action integral along each component diverges. For instance, on $\La_{+,+}$, the fiber variable is
\begin{align*}
\x(x)=\sqrt{\frac{1+4\pi^2h^2x^2}{x^2-1}},
\end{align*}
so $\x(x)-2\pi h=O_h(x^{-2})$ as $x\to\infty$. Thus the ordinary action integral along $\La_{+,+}$ diverges linearly, whereas subtracting the limiting value gives the convergent integral
\begin{align}\label{eq:renormalized-action}
\int_1^\infty\biggl(\sqrt{\frac{1+4\pi^2h^2x^2}{x^2-1}}-2\pi h\biggr)dx
=2\pi h I\biggl(1+\frac{1}{4\pi^2h^2}\biggr).
\end{align}
We call this the \emph{renormalized action} associated with $\La_{+,+}$. The corresponding quantities on the other components are given by the same formula with the appropriate sign.

\subsection{Notation}
 For $k\in\mathbb R$, $b\in S^k$ means that $b\in C^\infty(\mathbb R\times[0,1]_h)$ and
\begin{align*}
|\partial_\xi^\alpha b(\xi,h)|\lesssim\langle\xi\rangle^{k-\alpha},
\quad \alpha\in\mathbb N_0,
\end{align*}
uniformly in $\xi$ and $h$. We use the same notation on subsets of the $\xi$-line. We write $f=O_{\mathcal S(\Omega)}(h^\infty)$ if, for every $N\geq0$ and $\alpha,\beta\in\mathbb N_0$,
\begin{align*}
\sup_{x\in\Omega}\langle x\rangle^\beta|\partial_x^\alpha f(x;h)|\leq C_{N\alpha\beta}h^N.
\end{align*}
The notation $O(E^{-\infty})$ has the analogous meaning. For non-negative quantities $A=A(h)$ and $B=B(h)$, we write $A\sim B$ if $C^{-1}B\leq A\leq CB$ for some constant $C>0$ independent of $h$; when used for a formal series, $\sim$ has its usual asymptotic-expansion meaning.

\subsection{Boundary data}\label{subsection:boundary-data}

For later reference, we write the domain $D(P_{\rm sa})$ of $P_{\rm sa}$ explicitly. We set $u_{\rm even}(x)=\frac{1}{2}(u(x)+u(-x))$ and $u_{\rm odd}(x)=\frac{1}{2}(u(x)-u(-x))$. Then $D(P_{\rm sa})$ consists of elements of the maximal domain
\[
D_{\rm max}(P)=\{u\in L^2(\mathbb R)\mid Pu\in L^2(\mathbb R)\}
\]
which satisfy
\begin{align}
&\lim_{x\to\pm1}(x\mp1)\partial_xu(x)=0,\label{eq:bd1}\\
&\lim_{|x|\to\infty}\bigl(x\sin(2\pi x)\partial_xu_{\rm even}(x)
 -(2\pi x\cos(2\pi x)-\sin(2\pi x))u_{\rm even}(x)\bigr)=0,\label{eq:bd2}\\
&\lim_{|x|\to\infty}\bigl(x\cos(2\pi x)\partial_xu_{\rm odd}(x)
 +(2\pi x\sin(2\pi x)+\cos(2\pi x))u_{\rm odd}(x)\bigr)=0.\label{eq:bd3}
\end{align}
This direct description is more useful for our purposes than the Sonin space characterization outlined in the introduction.

Let $u\in D_{\rm max}(P)$. Boundary data are required at the two zeros $x=\pm1$ of the leading coefficient and at the two ends of the real line. Near $x=\s\in\{\pm\}$, the identity
\[
\partial_x\bigl((x^2-1)\partial_xu\bigr)
=-Pu-4\pi^2x^2u
\]
shows that $w=(x^2-1)\partial_xu$ belongs to $H^1_{\rm loc}$ and hence has a continuous representative. Since $x^2-1$ has a simple zero at $x=\s$, subtracting the logarithmic term determined by $w(\s)$ and using the one-dimensional Hardy inequality gives
\begin{align}\label{eq:max-domain-boundary-form}
u(x)=a_\s^{(0)}\log|x-\s|+a_\s^{(1)}H(x-\s)
+u_\s^{\rm reg}(x),
\quad u_\s^{\rm reg}\in H^1_{\rm loc}.
\end{align}
Condition \eqref{eq:bd1} is equivalent to $a_\s^{(0)}=0$.

At infinity, for fixed $z\in\mathbb C$, a basis of solutions of $(P-z)u=0$ on either exterior half-line has leading behavior
\[
\frac{\sin(2\pi x)}{|x|}+O_z(|x|^{-2}),
\quad
\frac{\cos(2\pi x)}{|x|}+O_z(|x|^{-2}),
\]
see \cite{CM}. Thus \eqref{eq:bd2} selects the sine-type asymptotic in the even sector, whereas \eqref{eq:bd3} selects the cosine-type asymptotic in the odd sector.

\begin{lem}\label{lem:asymxto1}
Let $\s_2\in\{\pm\}$ and $b\in S^{-1}$. Suppose that $b(\x,h)=|\x|^{-1}+O(|\x|^{-2})$ as $\s_2\x\to \infty$ and $\supp b(\cdot,h)\subset \{\x\mid \s_2 \x>0\}$. Then the function
\[
v_{\pm}(x)=\frac{1}{\sqrt{2\pi h}}\int_{\re}b(\x,h)e^{\frac{i}{h}(x\mp 1)\x}d\x
\]
satisfies $\lim_{x\to\pm 1}(x\mp 1)\pa_xv_{\pm}(x)=-1/\sqrt{2\pi h}$.
\end{lem}

\begin{proof}
By $(x\mp 1)e^{\frac{i}{h}(x\mp 1)\x}=hD_{\x}e^{\frac{i}{h}(x\mp 1)\x}$ and integration by parts, we have
\begin{align*}
\lim_{x\to \pm 1}(x\mp 1)\pa_xv_{\pm}(x)
&=-\frac{1}{\sqrt{2\pi h}}\lim_{x\to \pm 1}\int_{\re}\pa_{\x}(\x b(\x,h))e^{\frac{i}{h}(x\mp 1)\x}d\x\\
&=-\frac{1}{\sqrt{2\pi h}}\int_{\re}\pa_{\x}(\x b(\x,h))d\x\\
&=\left[-\frac{1}{\sqrt{2\pi h}}\x b(\x,h)\right]_{\x=-\infty}^{\x=\infty}=\frac{-1}{\sqrt{2\pi h}}.
\end{align*}
\end{proof}

\section{Model quasimodes and connection formula}\label{section:models}\label{section:EVnearBS}

\subsection{Local quasimode uniqueness}

We write $g(x)=x^2-1$, so
\begin{align*}
P(h)=-h^2\pa_x(g(x)\pa_x)-4\pi^2h^2x^2-1.
\end{align*}

We use the terminology of \cite[Defs.~E.35--E.36]{DZbook}. If $I\subset \re$ is an open interval, a family $u=u(h)\in\mathcal D'(I)$, $0<h<h_0$, is $h$-tempered if, for every $\chi\in C_{\rm c}^\infty(I)$, there are $C,N>0$ such that
\begin{align*}
\|\chi u\|_{H_h^{-N}(I)}\leq Ch^{-N}.
\end{align*}
Let $\Psi_h^m(I)$ denote the properly supported local semiclassical pseudodifferential calculus on the open interval $I$, obtained by localizing the usual uniform calculus on $\re$; see \cite[\S4.1]{Z} and \cite[\S E.1.2]{DZbook}. For $A\in\Psi_h^m(I)$, we denote its semiclassical principal symbol by $\sigma_{{\rm pr},h}(A)$. For an $h$-tempered family, $(x_0,\xi_0)\notin\WF_h(u)$ if there exists $A\in\Psi_h^0(I)$, elliptic at $(x_0,\xi_0)$, such that $Au=O_{C^\infty}(h^\infty)$.

We will need the following local uniqueness statement to match different quasimode constructions.

\begin{thm}\label{thm:WKBuniquethm}
Let $I\Subset (-\infty,-1)\cup (1,\infty)$ be connected and $\s\in\{\pm\}$. Suppose that $u=u(h),w=w(h)\in C^\infty(I)$ are $h$-tempered and satisfy
\begin{align*}
P(h)u=O_{C^{\infty}(I)}(h^{\infty}),\quad P(h)w=O_{C^{\infty}(I)}(h^{\infty}),\quad \WF_h(u)\cup\WF_h(w)\subset \{\s\x>0\}
\end{align*}
and $u(x_0)=w(x_0)+O(h^{\infty})$ at a point $x_0\in I$. Then we have $u(x)=w(x)+O_{C^{\infty}(I)}(h^{\infty})$.
\end{thm}

This result is a particular case of uniqueness of microlocal solutions for operators of real principal type; compare \cite[Lem.~18]{CP} and \cite[around (3.10)]{SZ}. In the special case of our one-dimensional problem, we give a direct proof using a local
WKB basis and the Wronskian. By symmetry it is enough to work with $\x>0$ and $I\Subset(1,\infty)$.

\begin{lem}
Let $I\Subset (1,\infty)$ and $x_0\in I$. There are real-valued phases $S_{\pm}\in C^{\infty}((1,\infty))$ with $\pm S_\pm'>0$, and classical symbols $
\z_{\pm}(x;h)\sim \sum_{j\geq0}h^j\z_{\pm,j}(x)$
such that, setting $\c_{\pm}(x)=\z_{\pm}(x;h)e^{\frac{i}{h}S_{\pm}(x)}$, 
\begin{align*}
P(h)\c_{\pm}=O_{C^{\infty}(I)}(h^{\infty}),\quad S_{\pm}(x_0)=0,\quad \z_{\pm}(x_0)=1+O(h^{\infty}).
\end{align*}
In particular, $\c_{\pm}(x_0)=1+O(h^\infty)$.
\end{lem}


\begin{proof}
This lemma is well-known and the proof is based on the standard WKB construction. We omit the details.
\end{proof}

\begin{lem}\label{lem:WKBuniWro}
Let $I\Subset (1,\infty)$ and $x_0\in I$. If an $h$-tempered family $u\in C^{\infty}(I)$ satisfies
\begin{align*}
P(h)u=O_{C^{\infty}(I)}(h^{\infty}),
\end{align*}
then there exist coefficients $c_{\pm} = c_{\pm}(h) \in \mathbb{C}$ such that
\begin{align*}
u(x)=c_{+}\c_+(x) +c_-\c_{-}(x) +O_{C^{\infty}(I)}(h^{\infty})
\end{align*}
for $x\in I$. 
Moreover, we can take $c_{\pm}$ explicitly as
\begin{align}\label{eq:Wronskian-coefficients}
c_+=\frac{W_h(u,\c_-)(x_0)}{W_h(\c_+,\c_-)(x_0)},\quad c_-=-\frac{W_h(u,\c_+)(x_0)}{W_h(\c_+,\c_-)(x_0)}.
\end{align}

\end{lem}

\begin{proof}
We introduce the weighted Wronskian
\begin{align*}
W_h(f_1,f_2)=g(x)\bigl(f_1h\pa_xf_2-(h\pa_xf_1)f_2\bigr).
\end{align*}
Then
\begin{align}\label{eq:Wronskian-identity}
h\pa_xW_h(f_1,f_2)=(P(h)f_1)f_2-f_1P(h)f_2.
\end{align}
We apply this identity with $(f_1,f_2)=(u,\c_\pm)$. Using temperedness we get that $W_h(u,\c_\pm)$ is constant modulo $O(h^\infty)$ on $I$, while $W_h(\c_+,\c_-)(x_0)$ is bounded away from zero for small $h$ since $S_+'(x_0)\neq S_-'(x_0)$. By applying Cramer's rule we obtain the formulas for coefficients, and hence the asserted expansion.
\end{proof}

\begin{proof}[Proof of Theorem \ref{thm:WKBuniquethm}]
We consider the case $\s=+$ only.
By Lemma \ref{lem:WKBuniWro}, we have
\begin{align*}
u(x)=c_+\c_+(x)+c_-\c_-(x)+O_{C^{\infty}(I)}(h^{\infty}),\,\,
w(x)=d_+\c_+(x)+d_-\c_-(x)+O_{C^{\infty}(I)}(h^{\infty})
\end{align*}
for $x\in I$.
From the coefficient formulas \eqref{eq:Wronskian-coefficients}, applied to $u$ and $w$, we deduce that $|c_{\pm}|+|d_{\pm}|\lesssim h^{-M}$ for some $M>0$. By the assumption $\WF_h(u)\cup \WF_h(w)\subset \{(x,\xi)\mid \x>0\}$ and $\WF_h(\c_-)\subset \{(x,\xi)\mid \x<0\}$, we obtain $c_-=d_-=O(h^{\infty})$. Hence
\begin{align*}
u(x)=c_+\c_+(x)+O_{C^{\infty}(I)}(h^{\infty}),\,\,
w(x)=d_+\c_+(x)+O_{C^{\infty}(I)}(h^{\infty}).
\end{align*}
Since $\c_{\pm}(x_0)=1+O(h^\infty)$, we have $u(x_0)=c_++O(h^{\infty})$ and $w(x_0)=d_++O(h^{\infty})$. Now our assumption implies $u(x_0)=w(x_0)+O(h^{\infty})$ and hence $c_+=d_++O(h^{\infty})$, which implies $u=w+O_{C^{\infty}(I)}(h^{\infty})$. This completes the proof.
\end{proof}

\subsection{Phases and model quasimodes}\label{subsection:phases-model-quasimodes}

 We define 
\begin{align}
\f_{\s_1,\s_2}(x)
&:=\s_2\biggl(\int_{\s_1\infty}^x
\biggl(\sqrt{\frac{1+4\pi^2h^2y^2}{y^2-1}}-2\pi h\biggr)dy+2\pi hx\biggr),
&&\s_1x>1,\label{eq:phase1}\\
F_{\s_1,\s_2}(\x)
&:=\s_1\biggl(\int_{\s_2 \infty}^{\x}
\biggl(\sqrt{\frac{\y^2+1}{\y^2-4\pi^2h^2}}-1\biggr)d\y+\x\biggr),
&&\s_2\x>2\pi h.\label{eq:phase2}
\end{align}
Note that $\La_{\s_1,\s_2}$, $\f_{\s_1,\s_2}$, and $F_{\s_1,\s_2}$ depend on $h$ (despite the notation not indicating this explicitly).

Geometrically, $\f_{\s_1,\s_2}$, and $F_{\s_1,\s_2}$  generate the component $\La_{\s_1,\s_2}$ introduced in \eqref{eq:characteristic-components} in the following sense:
\begin{align}
\La_{\s_1,\s_2}
&=\left\{(x,\pa_x\f_{\s_1,\s_2}(x))\mid \s_1x>1\right\}
\label{Lagmfddef}\\
&=\left\{(\pa_{\x}F_{\s_1,\s_2}(\x),\x)\mid \s_2\x>2\pi h\right\}.
\nonumber
\end{align}
The first parametrization is adapted to $|x|\to\infty$, while the second is adapted to approaching the radial point $\mathcal L_{\s_1,\s_2}$, i.e., to the regime $x\to\s_1$ and $|\x|\to\infty$.

\begin{defn}\label{def:symbolclasses}
Let $k,\ell\in \re$. We write $X=1+|x|$, $Y=1+h|x|$, and $L_h(x):=1+\log(Y/(hX))$. For $a\in C^{\infty}(\re\times (0,1]_h)$, we say that $a\in \mathcal S_{k}^{\ell}$ if for all $\a,\b\in\mathbb{N}_0$,
\begin{align*}
|\pa_x^{\a}\pa_h^{\b}a(x,h)|\lesssim 
h^{-\b}X^{-\frac12-\alpha}Y^{k}
L_h(x)^{\ell}
\end{align*}
uniformly in $x\in \re$ and $0< h\leq 1$. We write $a\in \mathcal S_{k}^{\ell}$ for $x\in \Omega$ with $\Omega\subset \re$ when $a$ satisfies these estimates uniformly on $x\in \Omega$ and $0< h\leq 1$.
\end{defn}

%

\begin{prop}\label{prop:quasimode}
For $\s_1,\s_2\in\{\pm\}$, there exist functions $u_{\s_1,\s_2}\in L^2(\re)$ satisfying
\begin{align}\label{eq:quasimodeeq}
P(h)u_{\s_1,\s_2}=O_{\mathcal{S}(\re)}(h^{\infty})
\end{align}
 as well as the following properties.\smallskip

\noindent$(i)$ For $\s_1x\geq2$,
\begin{align}\label{eq:quasimodeWKB}
u_{\s_1,\s_2}(x)
=\mathfrak{c}_{\s_1,\s_2}(h)a_{\s_1,\s_2}(x,h)e^{\frac{i}{h}\f_{\s_1,\s_2}(x)}
+O_{\mathcal{S}}(h^{\infty}).
\end{align}
For $\frac12\leq\s_1x\leq4$,
\begin{align}\label{eq:quasimodeLagrangian}
u_{\s_1,\s_2}(x)
=\frac{1}{\sqrt{2\pi h}}\int_{\re}b_{\s_1,\s_2}(\x,h)
e^{\frac{i}{h}(x\x-F_{\s_1,\s_2}(\x))}d\x
+O_{C^{\infty}}(h^{\infty}).
\end{align}
Moreover, $u_{\s_1,\s_2}(x)=0$ when $\s_1x\leq\frac14$. Above, $a_{\s_1,\s_2}\in \mathcal{S}_{-\frac{1}{2}}^{0}$ and $b_{\s_1,\s_2}\in S^{-1}$. There are constants $0<\x_0<\x_1$, independent of $h$, such that
\begin{align*}
\supp b_{\s_1,\s_2}(\cdot,h)\subset\{\x\in\re\mid \s_2\x\geq\x_0\},
\end{align*}
and
\begin{align}
&a_{\s_1,\s_2}(x,h)
-\frac{1}{|g(x)\f_{\s_1,\s_2}'(x)|^{\frac{1}{2}}}
\in h\mathcal{S}_{-\frac{3}{2}}^1,
\nonumber\\
&b_{\s_1,\s_2}(\x,h)
-|(\pa_xp_h)(F_{\s_1,\s_2}'(\x),\x)|^{-\frac{1}{2}}
\in hS^{-2},
\quad \s_2\x\geq\x_1,
\nonumber\\
&\mathfrak{c}_{+,+}(h)=\frac{1}{\sqrt{2}}
e^{2\pi iI(1+\frac{1}{4\pi^2h^2})+\frac{\pi i}{4}}
(1+r_{+,+}(h)),
\label{eq:connformula}\\
&\mathfrak{c}_{+,-}=\mathfrak{c}_{-,+}
=\overline{\mathfrak{c}_{-,-}}
=\overline{\mathfrak{c}_{+,+}}.
\nonumber
\end{align}
Here $r_{+,+}$ satisfies $|\pa_h^{\b}r_{+,+}(h)|\lesssim h^{1-\b}\log(1/h)$, and we recall that $I$ is defined in \eqref{eq:Iadef}.

\noindent$(ii)$ $\|u_{\s_1,\s_2}\|_{L^2(\re)}^2\sim\log(1/h)$.
\end{prop}

\begin{rem}
\noindent$(1)$  Microlocally away from the two ends of $\La_{\s_1,\s_2}$,
$u_{\s_1,\s_2}$ is a semiclassical Lagrangian distribution associated
with $\La_{\s_1,\s_2}$. Near $x=\s_1$, and for fixed $h$, it is conormal to $\{x=\s_1\}$ microlocally in the region $\s_2\x>0$. However, we stress that it is not globally described by a single standard semiclassical or scattering Lagrangian class.

\noindent$(2)$ One has $u_{\s_1,\s_2}\in H^{\frac12-0}(\re)\cap C^\infty(\re\setminus\{|x|=1\})$, but $u_{\s_1,\s_2}\notin H^{\frac12}_{\rm loc}(\re)$. 
\end{rem}

\proofsubtitle{Proof of Proposition \ref{prop:quasimode}.} Here, we construct $u_{+,+}$. The other $u_{\s_1,\s_2}$ can be constructed by setting
\begin{align*}
u_{+,-}(x):=\overline{u_{+,+}(x)},\quad u_{-,+}(x):=\overline{u_{+,+}(-x)},\quad u_{-,-}(x)=u_{+,+}(-x),
\end{align*}
where we have used that both the conjugation $u\mapsto \overline{u}$ and the reflection $u(x)\mapsto u(-x)$ commute with $P(h)$. In the following, we suppress the subscript $(+,+)$ when no risk of confusion arises.

\proofstep[Lagrangian distribution representation near $x=1$] We construct an ansatz on $\{\frac{1}{2}\leq x\leq 4\}$ in terms of a symbol
$b(\x,h)\sim \sum_{j=0}^{\infty}h^jb_j(\x,h)$ with $b_j\in S^{-1-j}$, supported away from the singular set $\x=2\pi h$.
The function $F=F_{+,+}$ satisfies the eikonal equation $p_h(F'(\x),\x)=0$, which is the graph parametrization of the component $\La_{+,+}$ in \eqref{Lagmfddef}. We set
\begin{align*}
J(b)(x):=\frac{1}{\sqrt{2\pi h}}
\int_{\re}e^{\frac{i}{h}(x\x-F(\x))}b(\x,h)d\x .
\end{align*}
We verify directly the transport formula used below; compare
\cite[Ex.~12.3]{GS} and \cite[Vol.~IV, Thm.~25.2.4]{H} for the general
calculus. Let $\mathcal F_h$ denote the unitary semiclassical Fourier transform
and put $c=2\pi h$. Then
\[
\mathcal F_hP(h)\mathcal F_h^{-1}
=-h^2\partial_\x\bigl((\x^2-c^2)\partial_\x\bigr)-\x^2-1.
\]
Conjugating this operator by $e^{-iF/h}$ and using
$p_h(F'(\x),\x)=0$, we obtain the exact identity
\[
P(h)J(b_j)=ihJ(\mathcal L_{H_{p_h}}b_j)+h^2J(r_j),
\quad r_j=-\partial_\x\bigl((\x^2-c^2)\partial_\x b_j\bigr),
\]
where
\[
\mathcal L_{H_{p_h}}b_j
=\alpha b_j'+\tfrac12\alpha'b_j,
\quad\alpha(\x)=\partial_xp_h(F'(\x),\x).
\]
In particular, $r_j\in S^{-1-j}$ and
$\supp r_j\subset\supp b_j$.
Therefore, for each $N>0$,
\begin{align*}
P(h)J\biggl(\sum_{j=0}^Nh^jb_j\biggr)(x)
={}&ihJ(\mathcal{L}_{H_{p_h}}b_0)+\sum_{j=0}^{N-1}h^{j+2}
\bigl(iJ(\mathcal{L}_{H_{p_h}}b_{j+1})+J(r_j)\bigr)\\
&+h^{N+2}J(r_N)+O_{C^{\infty}}(h^{\infty}).
\end{align*}
Thus, if we solve the equations $\mathcal{L}_{H_{p_h}}b_0=0$ and $i\mathcal{L}_{H_{p_h}}b_{j+1}+r_j=0$ successively, then we obtain a quasimode of $P(h)$ on $\frac{1}{2}\leq x\leq 4$. This will be used later in Lemma \ref{lem:quasiortho} and Lemma \ref{lem:quasimodeboundary}. In our case, after fixing the cutoff below, we solve these equations on the region $\x>\x_0$ concretely as
\begin{align}
b_0(\x)&=\chi(\x)|(\pa_xp_h)(F'(\x),\x)|^{-\frac{1}{2}},\label{eq:b_0}\\ 
b_{j+1}(\x)&=i\chi(\x)|(\pa_xp_h)(F'(\x),\x)|^{-\frac{1}{2}}
\int_{\infty}^{\x}|(\pa_xp_h)(F'(\y),\y)|^{-\frac{1}{2}}
 r_j(\y,h)d\y,
\label{eq:b_j}
\end{align}
where we write $b_j(\x)=b_j(\x,h)$; the symbolic properties are proved in Lemma \ref{lem:transport}. Here, we choose $\chi$ as follows. Let
\begin{align*}
\xi(x):=\sqrt{\frac{1+4\pi^2h^2x^2}{x^2-1}},
\end{align*}
so that $x=F'(\xi(x))$ for $x>1$. Since $F'$ is strictly decreasing on $(2\pi h,\infty)$ and $\xi(4)>1/4$ for $h>0$ small, we may fix
\begin{align*}
\x_0=\frac15,
\quad
\x_1=\frac14,
\quad
2\pi h<\x_0<\x_1<\xi(4)<\xi(2).
\end{align*}
We take $\chi\in C^\infty(\re;[0,1])$ with $\chi(\xi)=0$ for $\xi\leq \x_0$ and $\chi(\xi)=1$ for $\xi\geq \x_1$. Then the stationary points for $x\in[2,4]$ lie in the region where $\chi=1$. Moreover, on $\supp \chi'$ one has $F'(\xi)>4$, and hence $|x-F'(\xi)|\geq F'(\x_1)-4>0$ for $1/2\leq x\leq4$. Thus all terms produced by differentiating $\chi$ are $O_{C^\infty([1/2,4])}(h^\infty)$ by non-stationary phase. Similarly, if $x\in[1/2,2/3]$ and $\xi\geq \x_1$, then $F'(\xi)\geq1$, so $|x-F'(\xi)|\geq1/3$. We observe that
\begin{align}\label{eq:I(b)small}
J(b)(x)=O_{C^{\infty}}(h^{\infty})\quad \text{for}\,\, x\in \left[\frac{1}{2},\frac{2}{3}\right]
\end{align}
by the non-stationary phase argument. This estimate will be used in the final gluing of the two local expressions.

\proofstep[WKB representation at infinity] We construct a WKB ansatz on $\{x\geq2\}$; the uniform estimates through the scale $x\sim h^{-1}$ are proved in Lemma \ref{lem:log-symbol-aj}. Since $\f=\f_{+,+}$ satisfies the eikonal equation $p_h(x,\f'(x))=0$,
\begin{align*}
e^{-\frac{i}{h}\f(x)}P(h)(a(x)e^{\frac{i}{h}\f(x)})
={}&-ih\bigl(2g(x)\f'(x)a'(x)+g(x)\f''(x)a(x)\\
&\hspace{35mm}+g'(x)\f'(x)a(x)\bigr)\\
&-h^2\bigl(g(x)a''(x)+g'(x)a'(x)\bigr).
\end{align*}
where we write $a(x)=a(x,h)\sim \sum_{j=0}^{\infty}h^ja_j(x,h)$ with $a_j\in \mathcal{S}_{-\frac{1}{2}-j}^j$. We solve the transport equations
\begin{align*}
2ig(x)\f'(x)a_0'(x)
+ig(x)\f''(x)a_0(x)
+ig'(x)\f'(x)a_0(x)&=0,\\
2ig(x)\f'(x)a_{j+1}'(x)
+ig(x)\f''(x)a_{j+1}(x)
&=-ig'(x)\f'(x)a_{j+1}(x)\\
&\quad -g(x)a_j''(x)-g'(x)a_j'(x).
\end{align*}
successively as
\begin{equation}
\label{eq:WKBa_0}
a_0(x)=\frac{1}{|g(x)\f'(x)|^{\frac{1}{2}}}
=\frac{1}{(x^2-1)^{\frac{1}{4}}(4\pi^2h^2x^2+1)^{\frac{1}{4}}}
\in \mathcal{S}_{-\frac{1}{2}}^0.
\end{equation}
\begin{align}
a_j(x)
&=\frac{i}{2}|g(x)\f'(x)|^{-\frac{1}{2}}
\int_{\infty}^x
\frac{g(y)a_{j-1}''(y)+g'(y)a_{j-1}'(y)}{|g(y)\f'(y)|^{\frac{1}{2}}}
\,dy
\in \mathcal{S}_{-\frac{1}{2}-j}^j.
\label{eq:WKBa_j}
\end{align}
for $x\geq 2$, $j\geq 1$. The symbolic estimates of $a_j$ are proved in Lemma \ref{lem:log-symbol-aj}. In particular, the estimate $a(x)-a_0(x)\in h\mathcal{S}_{-\frac{3}{2}}^1$ uses the uniform boundedness of
\[
\frac{hL_h(x)}{Y}
=\frac{h}{1+h|x|}\biggl(1+\log\frac{1+h|x|}{h(1+|x|)}\biggr),
\quad 0<h\leq1,\quad |x|\geq2.
\]
Then, we obtain a quasimode of $P(h)$ on $\{x\geq 2\}$, that is, $P(h)(a(x)e^{\frac{i}{h}\f(x)})=O_{\mathcal{S}(\re)}(h^{\infty})$ for $x\geq 2$.

\proofstep[Matching on $2\leq x\leq4$] To match these two quasimodes on the region $2\leq x\leq 4$, we rewrite $J(b)(x)$ by using the stationary phase formula \cite[Thm.~3.11]{Z} as
\begin{align}\label{eq:stationary-phase-match}
J(b)(x)
&=e^{\frac{i}{h}(x\x(x)-F(\x(x)))+\frac{\pi i}{4}\sgn (-(F'')(\x(x))) }
\tilde{a}(x,h)
+O_{C^{\infty}}(h^{\infty})
\quad \text{for }x\in [2,4].
\end{align}
Here $\x(x)$ is the unique solution to $x=F'(\x(x))$ with $\x(x)>0$. We denote the leading amplitude by
\[
\widetilde a_0(x,h):=b_0(\x(x),h)|F''(\x(x))|^{-\frac12}.
\]
Then $\widetilde a$ is smooth for $x\in[2,4]$ and $h\in[0,1]$, and
\[
\widetilde a(x,h)=\widetilde a_0(x,h)+O_{C^\infty}(h)
\quad\text{for }x\in[2,4].
\]
We now compute the leading amplitude and phase more explicitly. By \eqref{eq:phase2}, we see $\sgn(-(F'')(\x(x)))=1$, where we recall $F=F_{+,+}$.
By a direct calculation, we see that the equation $x=F'(\x(x))$ is equivalent to $\x(x)=\f'(x)$. By differentiating $p_h(F'(\x),\x)=0$ with respect to $\x$, we have $F''(\x)=-(\pa_{\x}p_h)(F'(\x),\x)\cdot (\pa_{x}p_h)(F'(\x),\x)^{-1}$ and hence
\begin{align}\label{eq:tildeaexp}
\tilde{a}(x,h)
&=|(\pa_{\x}p_h)(x,\f'(x))|^{-\frac{1}{2}}+O_{C^{\infty}}(h)\nonumber\\
&=\frac{1}{\sqrt{2}|g(x)\f'(x)|^{\frac{1}{2}}}+O_{C^{\infty}}(h),
\quad x\in[2,4],
\end{align}
since $\chi=1$ on the stationary region. Finally, we observe that
\begin{align*}
\pa_x(x\x(x)-F(\x(x))-\f(x))&=\x(x)+x\x'(x)-\x'(x)F'(\x(x))-\f'(x)\\
&=\x'(x)(x-F'(\x(x)))=0
\end{align*}
by $\x(x)=\f'(x)$ and $x=(F')(\x(x))$, which implies $x\x(x)-F(\x(x))-\f(x)$ is independent of $x$. Since $\x(x)=\f'(x)\to2\pi h$ as $x\to\infty$ and
\[
\sqrt{\frac{1+4\pi^2h^2y^2}{y^2-1}}-2\pi h=O_h(|y|^{-2}),
\]
we have $x\x(x)-\f(x)\to0$ as $x\to\infty$. Therefore,
\begin{align*}
\lim_{x\to \infty}\bigl(x\x(x)-F(\x(x))-\f(x)\bigr)
&=-\lim_{x\to \infty}F(\x(x))\\
&=-\int_{\infty}^{2\pi h}
\biggl(\sqrt{\frac{\y^2+1}{\y^2-4\pi^2h^2}}-1\biggr)d\y-2\pi h\\
&=\int^{\infty}_{2\pi h}
\biggl(\sqrt{\frac{\y^2+1}{\y^2-4\pi^2h^2}}-1\biggr)d\y-2\pi h\\
&=2\pi h I\biggl(1+\frac{1}{4\pi^2h^2}\biggr)-2\pi h,
\end{align*}
where we recall that $I$ is defined in \eqref{eq:Iadef}. Since $e^{-2\pi i}=1$, the term $-2\pi h$ does not contribute after division by $h$ and exponentiation. Consequently,
\[
e^{\frac{i}{h}(x\x(x)-F(\x(x)))}
=e^{\frac{i}{h}\f(x)+2\pi i I(1+\frac{1}{4\pi^2h^2})}.
\]
In summary, we obtain
\begin{align*}
J(b)(x)
&=e^{\frac{i}{h}\f(x)+2\pi iI(1+\frac{1}{4\pi^2h^2})+\frac{\pi}{4}i}
\tilde{a}(x,h)+O_{C^{\infty}}(h^{\infty})
\quad \text{for }x\in [2,4].
\end{align*}

Now we take $\mathfrak{c}_{+,+}=\mathfrak{c}_{+,+}(h)$ as
\begin{align*}
\mathfrak{c}_{+,+}(h)
:=a(3,h)^{-1}\tilde{a}(3,h)
e^{2\pi iI(1+\frac{1}{4\pi^2h^2})+\frac{\pi}{4}i}.
\end{align*}
By \eqref{eq:WKBa_0}, \eqref{eq:WKBa_j}, \eqref{eq:tildeaexp}, and the uniqueness of microlocal solutions (Theorem \ref{thm:WKBuniquethm}), we see that
\begin{align}
\mathfrak{c}_{+,+}(h)
&=\frac{1}{\sqrt{2}}
e^{2\pi iI(1+\frac{1}{4\pi^2h^2})+\frac{\pi}{4}i}
\bigl(1+r_{+,+}(h)\bigr),
\nonumber\\
J(b)(x)
&=\mathfrak c_{+,+}(h)a(x)e^{\frac{i}{h}\f(x)} +O_{C^{\infty}}(h^{\infty}).
\label{eq:twoWKBequal}
\end{align}
where $r_{+,+}$ satisfies $|\pa_h^{\b}r_{+,+}(h)|\lesssim h^{1-\b}\log(1/h)$.

\proofstep[Gluing] We define $K_1=[\frac{2}{3},2]$, $K'_1=[\frac{1}{2},4]$, $K_2=[4,\infty)$, and $K'_2=[2,\infty)$.
We take $\chi_1,\chi_2\in C^{\infty}(\re;[0,1])$ with $\chi_j=1$ on $K_j$, $\supp \chi_j\subset K_j'$ ($j=1,2$), and $\chi_1+\chi_2=1$ on $[2,4]$. We define
\begin{align*}
u_{+,+}(x)=\chi_1(x)J(b)(x)+\mathfrak{c}_{+,+}(h)\chi_2(x)a(x)e^{\frac{i}{h}\f(x)}.
\end{align*}
By \eqref{eq:I(b)small} and \eqref{eq:twoWKBequal}, we have $P(h)u_{+,+}=O_{\mathcal{S}(\re)}(h^{\infty})$. This completes the construction of $u_{+,+}$.

\proofstep[$L^2$ size] Since $|g(x)\f_{\s_1,\s_2}'(x)|^{\frac{1}{2}}\sim |x|^{\frac{1}{2}}(1+h|x|)^{\frac{1}{2}}$, we have
\begin{align*}
\|u_{\s_1,\s_2}\|_{L^2(\{\s_1x\geq2\})}^2
&\sim\int_{\s_1x\geq2}|x|^{-1}(1+h|x|)^{-1}dx
\sim\log(1/h),\\
\|u_{\s_1,\s_2}\|_{L^2(\{|x|\leq2\})}^2&=O(1).
\end{align*}
\qed\medskip

\subsection{Almost orthogonality} Finally, we show that the quasimodes constructed above are almost orthogonal to each other.

\begin{lem}[Almost orthogonality of quasimodes]\label{lem:quasiortho}
Let $u_{\s_1,\s_2}$ be as in Proposition \ref{prop:quasimode}. Then,
\begin{align*}
\langle u_{\s_1,\s_2},u_{\s_1',\s_2'}\rangle_{L^2}=O(h^{\infty})
\end{align*}
when $(\s_1,\s_2)\neq (\s_1',\s_2')$.
\end{lem}

\begin{proof}

It suffices to prove
\begin{align*}
 \langle u_{+,+},u_{+,-}\rangle_{L^2}=O(h^\infty),
\end{align*}
since the case $\langle u_{-,+},u_{-,-}\rangle_{L^2}$ then follows by applying the symmetry
$x\mapsto -x$. If $\s_1\neq \s_1'$, then the supports of the
corresponding quasimodes are disjoint modulo $O_{\mathcal S}(h^\infty)$,
and the conclusion is immediate. We therefore consider only
$\langle u_{+,+},u_{+,-}\rangle_{L^2}$.

\proofstep[Compact region] Let $\chi_0\in C^\infty_{\rm c}(\re;[0,1])$ satisfy
\[
\chi_0=0\quad\text{on }(-\infty,\tfrac12],
\quad
\chi_0=1\quad\text{on }[\tfrac23,3],
\quad
\supp\chi_0\subset(\tfrac12,4).
\]
On $\supp\chi_0$ we use the Lagrangian distribution representations from
\eqref{eq:quasimodeLagrangian}:
\begin{align*}
u_{+,+}(x)&=\frac1{\sqrt{2\pi h}}\int
 e^{\frac ih(x\x-F_{+,+}(\x))}
 b_{+,+}(\x,h)\,d\x+O_{C^\infty}(h^\infty),\\
u_{+,-}(x)&=\frac1{\sqrt{2\pi h}}\int
 e^{\frac ih(x\eta-F_{+,-}(\eta))}
 b_{+,-}(\eta,h)\,d\eta+O_{C^\infty}(h^\infty).
\end{align*}
By Proposition \ref{prop:quasimode},
\[
\supp b_{+,+}\subset\{\x\geq\x_0\},
\quad
\supp b_{+,-}\subset\{\eta\leq-\x_0\},
\]
so $|\x-\eta|\geq2\x_0$ on the support of the amplitudes. Thus
\begin{align*}
\langle \chi_0u_{+,+},u_{+,-}\rangle_{L^2}
&=\frac1{2\pi h}\iiint
 e^{\frac ih(x(\x-\eta)-F_{+,+}(\x)+F_{+,-}(\eta))}
 \chi_0(x)b_{+,+}(\x,h)\overline{b_{+,-}(\eta,h)}
\,dx\,d\x\,d\eta \\
&\quad +O(h^\infty).
\end{align*}
We integrate by parts in $x$ using
\[
 e^{\frac ihx(\x-\eta)}
 =
 \frac{h}{i(\x-\eta)}\pa_x
 e^{\frac ihx(\x-\eta)}.
\]
For every $N$, this gives
\begin{align*}
 |\langle \chi_0u_{+,+},u_{+,-}\rangle_{L^2}|
 \leq C_N h^{N-1}
 \iint
 \frac{|b_{+,+}(\x,h)b_{+,-}(\eta,h)|}
 {|\x-\eta|^N}\,d\x\,d\eta
 +O(h^\infty).
\end{align*}
Taking $N$ large, the last integral is uniformly bounded since
$b_{+,+},b_{+,-}\in S^{-1}$ and $|\x-\eta|\geq2\x_0$. Hence
\[
\langle \chi_0u_{+,+},u_{+,-}\rangle_{L^2}=O(h^\infty).
\]

\proofstep[WKB estimate at infinity] By the gluing construction and \eqref{eq:I(b)small}, the part of the complementary term with $x\leq\frac23$ is $O(h^\infty)$. It therefore remains to treat the region $x\geq3$, where we use the WKB representations \eqref{eq:quasimodeWKB}
\begin{align*}
u_{+,+}(x)=\mathfrak c_{+,+}(h)a_{+,+}(x,h)e^{\frac ih\f_{+,+}(x)}
+O_{\mathcal S}(h^\infty),
\end{align*}
\begin{align*}
u_{+,-}(x)=\mathfrak c_{+,-}(h)a_{+,-}(x,h)e^{\frac ih\f_{+,-}(x)}
+O_{\mathcal S}(h^\infty).
\end{align*}
Since $\f_{+,-}=-\f_{+,+}$, we can define
\begin{align*}
 \g(x):=\f_{+,+}(x)-\f_{+,-}(x)=2\f_{+,+}(x).
\end{align*}
Then
\begin{align*}
 \g'(x)
 =
 2\sqrt{\frac{1+4\pi^2h^2x^2}{x^2-1}},
 \quad x>1.
\end{align*}
In particular,
\begin{align*}
 \g'(x)\sim x^{-1},\quad 3\leq x\leq2/h,
 \quad
 \g'(x)\sim h,\quad x\geq1/h.
\end{align*}
Moreover, the amplitude estimates in Proposition \ref{prop:quasimode}, together with Leibniz' rule, give directly
\begin{align*}
 \pa_x^k\bigl(a_{+,+}(x,h)\overline{a_{+,-}(x,h)}\bigr)
 =O\bigl((1+hx)^{-1}x^{-1-k}\bigr),
 \quad x\geq3.
\end{align*}

We choose $\chi_2,\chi_3\in C^\infty(\re;[0,1])$, depending on $h$, such that
\begin{align*}
 \chi_2+\chi_3=1-\chi_0 \quad\text{on } x\geq 3,
\end{align*}
with
\begin{align*}
 \supp\chi_2\subset \{3\leq x\leq 2/h\},
 \quad
 \supp\chi_3\subset \{x\geq 1/h\},
\end{align*}
and with the usual bounds
$|\pa_x^k\chi_2|+|\pa_x^k\chi_3|\leq C_kx^{-k}$ on their supports. Since the connection coefficients in \eqref{eq:quasimodeWKB} are uniformly bounded, it suffices to estimate
\begin{align*}
 I_j:=\int e^{\frac ih\g(x)}
 \chi_j(x,h)a_{+,+}(x,h)\overline{a_{+,-}(x,h)}\,dx,
 \quad j=2,3.
\end{align*}

\proofsubtitle{Estimate of $I_2$.} First consider $I_2$. On $\supp\chi_2$, we have
$\g'(x)\sim x^{-1}$ and
\begin{align*}
 \pa_x^k(\g'(x)^{-1})=O(x^{1-k}).
\end{align*}
Let 
 $L:=\frac{h}{i\g'(x)}\pa_x$, so that $Le^{\frac ih\g}=e^{\frac ih\g}$.
If $\a_2:=\chi_2a_{+,+}\overline{a_{+,-}}$ 
then repeated integration by parts gives
\begin{align*}
 I_2=\int e^{\frac ih\g}(L^t)^N\a_2\,dx.
\end{align*}
The symbol bounds above imply $(L^t)^N\a_2=O(h^N x^{-1})$
on $3\leq x\leq 2/h$. Consequently
\begin{align*}
 |I_2|
 \leq C_N h^N\int_3^{2/h}\frac{dx}{x}
 \leq C_N h^N\log(1/h)
 =O(h^\infty),
\end{align*}
since $N$ is arbitrary.

\proofsubtitle{Estimate of $I_3$.} Now consider $I_3$. On $\supp\chi_3$, we have
$\g'(x)\sim h$, and
\begin{align*}
 \pa_x^k\biggl(\frac{h}{\g'(x)}\biggr)=O(x^{-k}).
\end{align*}
Furthermore, since $hx\geq 1$,
\begin{align*}
 \pa_x^k\a_3
 =
 O(h^{-1}x^{-2-k}),
 \quad
 \a_3:=\chi_3a_{+,+}\overline{a_{+,-}}.
\end{align*}
For the same operator  $L$ as before, we obtain $(L^t)^N\a_3=O(h^{-1}x^{-2-N})$.
Therefore by integration by parts we get
\begin{align*}
 |I_3|
 \leq C_N h^{-1}\int_{1/h}^\infty x^{-2-N}\,dx
 \leq C_N h^N.
\end{align*}
Again $N$ is arbitrary, hence $I_3=O(h^\infty)$.

Combining the compact-region estimate with the two WKB arguments gives
\begin{align*}
 \langle u_{+,+},u_{+,-}\rangle_{L^2}=O(h^\infty).
\end{align*}
The same argument after the change of variables $x\mapsto -x$ gives
$ \langle u_{-,+},u_{-,-}\rangle_{L^2}=O(h^\infty)$.
This proves the lemma.
\end{proof}

In view of Proposition \ref{prop:quasimode}(ii), Lemma \ref{lem:quasiortho} also shows that the four model quasimodes are linearly independent for all sufficiently small $h$.

\subsection{Boundary values}

\begin{lem}\label{lem:quasimodeboundary}
Let $u=u_{\s_1,\s_2}$, and write $u_{\rm even}$ and $u_{\rm odd}$ for its even and odd parts. Then
\begin{align*}
&\lim_{x\to\pm\infty}\Bigl(
 x\sin(2\pi x)\pa_xu_{\rm even}(x)
 -(2\pi x\cos(2\pi x)-\sin(2\pi x))u_{\rm even}(x)
 \Bigr)
 =\mp\sqrt{\frac{\pi}{2h}}\mathfrak c_{\s_1,\s_2},\\
&\lim_{x\to\pm\infty}\Bigl(
 x\cos(2\pi x)\pa_xu_{\rm odd}(x)
 +(2\pi x\sin(2\pi x)+\cos(2\pi x))u_{\rm odd}(x)
 \Bigr)
 =\pm i\s_2\sqrt{\frac{\pi}{2h}}\mathfrak c_{\s_1,\s_2},\\
&\lim_{x\to\s}(x-\s)\pa_xu_{\s_1,\s_2}(x,h)
 =-\frac{\delta_{\s,\s_1}}{2\sqrt{\pi h}},
 \quad \s\in\{\pm\}.
\end{align*}
\end{lem}

\begin{proof}
Observe that $\f_{\s_1,\s_2}(x)=2\pi h\s_2x+O_h(|x|^{-1})$ as $\s_1x\to\infty$. By \eqref{eq:quasimodeWKB} and the definitions of $g$ and $\f_{\s_1,\s_2}$, we have
\begin{align*}
u_{\s_1,\s_2}(x)
&=\left\{\begin{aligned}
&\frac{\mathfrak c_{\s_1,\s_2}}{\sqrt{2\pi h}}|x|^{-1}e^{2\pi i\s_2x}
 +O_h(|x|^{-2})&& (\s_1x\to\infty),\\
&O(|x|^{-\infty})&& (\s_1x\to-\infty),
\end{aligned}\right.\\
\pa_xu_{\s_1,\s_2}(x)
&=\left\{\begin{aligned}
&2\pi i\s_2\frac{\mathfrak c_{\s_1,\s_2}}{\sqrt{2\pi h}}|x|^{-1}e^{2\pi i\s_2x}
 +O_h(|x|^{-2})&& (\s_1x\to\infty),\\
&O(|x|^{-\infty})&& (\s_1x\to-\infty).
\end{aligned}\right.
\end{align*}
Consequently, as $|x|\to\infty$,
\begin{align*}
u_{\rm even}(x)
&=\frac{\mathfrak c_{\s_1,\s_2}}{2\sqrt{2\pi h}}|x|^{-1}
 e^{2\pi i\s_1\s_2|x|}+O_h(|x|^{-2}),\\
u_{\rm odd}(x)
&=\frac{\s_1\sgn(x)\mathfrak c_{\s_1,\s_2}}{2\sqrt{2\pi h}}|x|^{-1}
 e^{2\pi i\s_1\s_2|x|}+O_h(|x|^{-2}).
\end{align*}
Substitution in the two boundary functionals gives the first two limits.

For the $x\to \sigma$ limit, set
\[
q_{\s_1,\s_2}(\x)=F_{\s_1,\s_2}(\x)-\s_1\x.
\]
The explicit formula \eqref{eq:phase2} gives $q_{\s_1,\s_2}(\x)=O_h(|\x|^{-1})$ as $\s_2\x\to\infty$. Together with Proposition \ref{prop:quasimode}, this yields
\[
e^{-\frac{i}{h}q_{\s_1,\s_2}(\x)}b_{\s_1,\s_2}(\x,h)
=\frac{1}{\sqrt2|\x|}+O_h(|\x|^{-2}),
\quad \s_2\x\to\infty.
\]
Thus the Lagrangian distribution representation \eqref{eq:quasimodeLagrangian} can be written near $x=\s_1$ as
\[
u_{\s_1,\s_2}(x)
=\frac{1}{\sqrt{2\pi h}}\int_{\re}
\biggl(\frac{1}{\sqrt2|\x|}+O_h(|\x|^{-2})\biggr)
 e^{\frac{i}{h}(x-\s_1)\x}\,d\x
 +O_{C^\infty}(h^\infty),
\]
with the amplitude supported in $\{\s_2\x>0\}$. Lemma \ref{lem:asymxto1} implies then the stated limit at $x=\s_1$, while $u_{\s_1,\s_2}$ vanishes near the opposite degeneracy.
\end{proof}

\subsection{Phase correction}

The connection formula \eqref{eq:connformula} determines a real phase correction. Since $r_{+,+}(h)=O(h\log(1/h))$, there are unique smooth functions $\varrho(h)>0$ and $d(h)\to0$ such that
\begin{align}\label{eq:def-d}
1+r_{+,+}(h)=\varrho(h)e^{\pi i d(h)}.
\end{align}
The function $d$ satisfies $|\partial_h^\beta d(h)|\lesssim h^{1-\beta}\log(1/h)$. We set
\begin{align}\label{eq:BS-function}
\mathrm{BS}(h)=2I\biggl(1+\frac{1}{4\pi^2h^2}\biggr)+\frac14+d(h),
\end{align}
so that
\begin{align}\label{eq:connection-BS}
\mathfrak c_{+,+}(h)=\frac{\varrho(h)}{\sqrt2}e^{\pi i\mathrm{BS}(h)}.
\end{align}

The following calculation is the analogue of the first correction in the all-orders Bohr--Sommerfeld rule of \cite{C1}. The new feature is that the first transport integral has a borderline $x^{-1}$ term on $1\ll x\ll h^{-1}$, producing an $h\log(1/h)$ phase correction.

\begin{lem}[Logarithmic correction to Bohr--Sommerfeld remainder]\label{lem:first-log-correction-d} The function $d(h)$ defined by \eqref{eq:def-d} satisfies
\begin{align}\label{eq:first-log-d}
d(h)=-\frac{h}{8\pi}\log(1/h)+O(h),\quad h\to0.
\end{align}
\end{lem}

\begin{proof}
We isolate the only contribution to the connection coefficient which can be of size $h\log(1/h)$. We work on the $(+,+)$ component and put
\begin{align*}
c=2\pi h,\quad
G(x,h):=|g(x)\f'(x)|=((x^2-1)(1+c^2x^2))^{1/2},
\quad a_0=G^{-1/2}.
\end{align*}
From \eqref{eq:WKBa_j},
\begin{align}\label{eq:a1-over-a0}
\frac{a_1(x,h)}{a_0(x,h)}
=\frac{i}{2}\int_{\infty}^{x}Q(y,h)\,dy,
\quad
Q:=\frac{g a_0''+g'a_0'}{G^{1/2}}.
\end{align}
A direct calculation gives
\begin{align}\label{eq:Qexplicit}
Q(x,h)=
-\frac{
2c^4x^4-3c^4x^2+6c^2x^4-10c^2x^2+2c^2+x^2-2
}
{4(x^2-1)^{3/2}(1+c^2x^2)^{5/2}}.
\end{align}
For $3\le x\le c^{-1}$ this implies
\begin{align*}
Q(x,h)=-\frac{1}{4x}+O(x^{-3})+O(c^2x),
\end{align*}
and the integral of the two error terms over $[3,c^{-1}]$ is $O(1)$. On the remaining region $x\ge c^{-1}$, the change of variables $s=cx$ in \eqref{eq:Qexplicit} gives
\begin{align*}
\int_{c^{-1}}^\infty |Q(x,h)|\,dx=O(1).
\end{align*}
Thus
\begin{align}\label{eq:Qlog}
\int_{\infty}^{3}Q(y,h)\,dy
=\frac14\log(1/c)+O(1)
=\frac14\log(1/h)+O(1),
\end{align}
because $c=2\pi h$. Combining \eqref{eq:a1-over-a0} and \eqref{eq:Qlog}, we get
\begin{align}\label{eq:a1log}
\frac{a_1(3,h)}{a_0(3,h)}=\frac{i}{8}\log(1/h)+O(1).
\end{align}
The stationary phase expansion together with Lemma \ref{lem:transport} implies
\begin{align}\label{eq:atilde-no-log}
\widetilde a(3,h)=\widetilde a_0(3,h)(1+h\widetilde A_1(h)+O(h^2)),
\quad \widetilde A_1(h)=O(1).
\end{align}
Moreover, by the definition of $\widetilde a_0$ and the calculation leading to \eqref{eq:tildeaexp}, $\widetilde a_0(3,h)=2^{-1/2}a_0(3,h)$. Therefore, using \eqref{eq:a1log},
\begin{align}\label{eq:ratio-log}
\frac{\widetilde a(3,h)}{a(3,h)}
=\frac{1}{\sqrt2}\bigl(1-\frac{i h}{8}\log(1/h)+O(h)\bigr).
\end{align}
By the definition of $\mathfrak c_{+,+}$, this means that the factor $1+r_{+,+}(h)$ in \eqref{eq:connformula} satisfies
\begin{align}\label{eq:r-first-log}
1+r_{+,+}(h)=1-\frac{i h}{8}\log(1/h)+O(h).
\end{align}
We use the branch of the logarithm near $1$ and obtain
\begin{align*}
d(h)=\pi^{-1}{\rm Im}\log(1+r_{+,+}(h))
=-\frac{h}{8\pi}\log(1/h)+O(h),
\end{align*}
which proves \eqref{eq:first-log-d}. 
\end{proof}

\section{Completeness of the model quasimodes}\label{section:completeness}

\subsection{Completeness statement} The main objective of this section is to prove a completeness statement for the four model quasimodes constructed previously. We state the result first; Proposition \ref{prop:asymexpef} is then proved in \S\ref{subsection:proof-completeness} once several auxiliary results are established.

\begin{prop}\label{prop:asymexpef}
Let $u\in L^2(\mathbb R)$ satisfy $\|u\|_{L^2}=1$ and $P(h)u=0$. There are coefficients $c_{\s_1,\s_2}=c_{\s_1,\s_2}(h)\in\mathbb C$ such that
\begin{align}\label{eq:coefsizeasymp}
\sum_{\s_1,\s_2\in\{\pm\}}|c_{\s_1,\s_2}|^2\sim(\log(1/h))^{-1}
\end{align}
{and if} $v=u-\sum_{\s_1,\s_2\in\{\pm\}}c_{\s_1,\s_2}u_{\s_1,\s_2}$, then
\begin{align*}
v&=O_{\mathcal S(\{|x|\geq2\})}(h^\infty),\\
v&=O_{C^\infty(K)}(h^\infty)\quad\text{for every }K\Subset\mathbb R\setminus\{\pm1\},\\
v&=O_{H_h^{\frac12-\delta}(\mathbb R)}(h^\infty)\quad\text{for every }\delta>0.
\end{align*}
Moreover, for each $\s\in\{\pm\}$ there are $a_\s^{(0)}(h),a_\s^{(1)}(h)=O(h^\infty)$ such that, locally near $x=\s$,
\begin{align}\label{eq:log-Heaviside-representation}
v(x)=a_\s^{(0)}(h)\log|x-\s|+a_\s^{(1)}(h)H(x-\s)
+O_{H^{\frac32-\delta}}(h^\infty)
\end{align}
for every $\delta>0$. 
\end{prop}

\subsection{Polynomial bounds at infinity}\label{subsection:energyest}

We begin with polynomial bounds needed to justify the Wronskian decomposition at infinity. Since the leading coefficient $x^2-1$ does not vanish on $(1,\infty)$, every distributional solution of $P(h)u=0$ is smooth there. For $x>1$,
\begin{align*}
P(h)u=0\quad \Longleftrightarrow\quad h^2(gu')'=(V_h-1)u,
\end{align*}
where $V_h(x)=-4\pi^2h^2x^2$.

\begin{lem}\label{lem:semicenergyest}
Let $x_0\in (1,\infty)$. Suppose that $u=u_h\in C^{\infty}((1,\infty))$ satisfies $P(h)u=0$ and
\begin{align*}
|u(x_0)|\leq Ch^{-N},\quad |u'(x_0)|\leq Ch^{-N}
\end{align*}
for some $N\geq0$ and $C>0$ independent of $h$. Then, for every integer $k\geq0$, there exists $M_k>0$ such that
\begin{align*}
\|u^{(k)}\|_{L^{\infty}([x_0,\infty))}=O(h^{-M_k}).
\end{align*}
The analogous statement on $(-\infty,-x_0]$ follows by the symmetry $x\mapsto -x$.
\end{lem}

\begin{proof}
We define, for $x>1$, the energy functional
\begin{align*}
E(x):=|hu'(x)|^2+\frac{1-V_h(x)}{g(x)}|u(x)|^2 .
\end{align*}
Using that $h^2u''= -h^2\frac{g'}g u' + \frac{V_h-1}{g}u$ we compute
\begin{align*}
E'(x)
&=2\Re\!\big(\overline{u'(x)}h^2u''(x)\big)
+\biggl(\frac{1-V_h}{g}\biggr)'|u(x)|^2
+2\frac{1-V_h}{g}\Re\!(\overline{u(x)}u'(x))\\
&=-2h^2\frac{g'(x)}{g(x)}|u'(x)|^2
+\biggl(\frac{1-V_h}{g}\biggr)'|u(x)|^2 .
\end{align*}
Since
\begin{align*}
\frac{g'(x)}{g(x)}=\frac{2x}{x^2-1}>0,
\quad
\biggl(\frac{1-V_h(x)}{g(x)}\biggr)'
= -\frac{2(4\pi^2h^2+1)x}{(x^2-1)^2}<0
\end{align*}
for $x>1$, we have $E'(x)\leq0$. Thus $E(x)\leq E(x_0)$ for $x\geq x_0$. Since $x_0>1$,
\begin{align*}
E(x_0)\leq h^2|u'(x_0)|^2+C_0|u(x_0)|^2=O(h^{-2N}).
\end{align*}
Moreover,
\begin{align*}
\frac{1-V_h(x)}{g(x)}=\frac{1+4\pi^2h^2x^2}{x^2-1}\geq 4\pi^2h^2,
\quad x\geq x_0.
\end{align*}
 Hence
\begin{align*}
\|u\|_{L^{\infty}([x_0,\infty))}+\|u'\|_{L^{\infty}([x_0,\infty))}=O(h^{-N-1}).
\end{align*}
It remains to bound higher derivatives. We write
\begin{align*}
u''=Au'+Bu,
\quad
A(x)=-\frac{g'(x)}{g(x)},
\quad
B(x)=\frac{V_h(x)-1}{h^2g(x)}=B_0(x)+h^{-2}B_1(x),
\end{align*}
where
\begin{align*}
B_0(x)=-\frac{4\pi^2x^2}{x^2-1},
\quad
B_1(x)=-\frac{1}{x^2-1}.
\end{align*}
All derivatives of $A,B_0,B_1$ are bounded on $[x_0,\infty)$, consequently
\begin{align*}
\|B^{(\ell)}\|_{L^{\infty}([x_0,\infty))}\leq C_\ell h^{-2}.
\end{align*}
Differentiating $u''=Au'+Bu$ $k$ times gives
\begin{align*}
u^{(k+2)}=\sum_{j=0}^{k}\binom{k}{j}\bigl(A^{(j)}u^{(k-j+1)}+B^{(j)}u^{(k-j)}\bigr).
\end{align*}
 Inductively, starting from the bounds for $u$ and $u'$, we obtain
\begin{align*}
\|u^{(k)}\|_{L^{\infty}([x_0,\infty))}=O(h^{-M_k})
\end{align*}
for some exponents $M_k$, for example $M_{2m}=M_{2m+1}=N+1+2m$.
\end{proof}

\begin{lem}\label{lem:localelliptic}
Suppose that $u=u_h\in L^2(\re)$ satisfies $P(h)u=0$ and
\begin{align*}
\|u\|_{L^2}\leq Ch^{-M}
\end{align*}
for some fixed $M\geq0$ and $C>0$ independent of $h$. Then, for each $I\Subset(-\infty,-1)\cup(1,\infty)$ and $k\in\mathbb N_0$, there exists $M_k\geq0$ such that
\begin{align*}
\|u^{(k)}\|_{L^\infty(I)}\lesssim h^{-M_k}.
\end{align*}
\end{lem}

\begin{proof}
For fixed $k$, choose a sufficiently long nested sequence of relatively
compact intervals,
whose closures avoid $\{\pm1\}$. On these intervals, $|g|$ is bounded
away from zero. Standard interior elliptic estimates for the ordinary differential operator $-\partial_x(g\partial_x)$, together with
\begin{align*}
-\partial_x(g\partial_x)u=(4\pi^2x^2+h^{-2})u,
\end{align*}
give recursively
\begin{align*}
\|u\|_{H^{m+2}(I_m)}\leq C_mh^{-2}\|u\|_{H^m(I_{m+1})}.
\end{align*}
Thus $\|u\|_{H^{2r}(I)}\lesssim h^{-M-2r}$ for every fixed $r$, and one-dimensional Sobolev embedding gives the result; for example, one may take $M_k=M+2\lceil(k+1)/2\rceil$. In the application below, $\|u\|_{L^2}=1$, so $M=0$.
\end{proof}

\subsection{Radial estimate}\label{subsection:radialestimate}
Focusing for now our analysis on fixed compact subsets of the $x$-line, the relevant object is the semiclassical principal symbol of
$P(h)$, which we denote by
\[
p:=\sigma_{{\rm pr},h}(P(h))=(x^2-1)\x^2-1,
\]
and its characteristic set in $\overline{T^*\re}$ is
$\{\langle\x\rangle^{-2}p=0\}$.  
In the proofs  we will  use \cite[Thm.~E.54]{DZbook}, the \textit{below-threshold radial estimate} recalled below, applied to an operator of the form
\[
P_B(h):=P(h)-h^2B(h),\quad B(h)\in\Psi_h^0,
\]
at a radial sink, and to $-P_B(h)$ at a radial source. All neighborhoods introduced in the sequel are independent of $h$.

\begin{prop}[Below-threshold radial estimate]\label{prop:basic-radial-estimate}
Fix $\s_1,\s_2\in\{\pm\}$, let $s<1/2$, and let $B(h)$ be a bounded
family in $\Psi_h^0$. Suppose that $G_1\in\Psi_h^0$ is elliptic near
$\mathcal L_{\s_1,\s_2}$. For every sufficiently small neighborhood $V$ of
$\mathcal L_{\s_1,\s_2}$ in $\overline{T^*\re}$ with
$\overline V\subset\operatorname{ell}_h(G_1)$, there exist
$G,G_0\in\Psi_h^0$ and $\chi\in C^\infty_{\rm c}(\re)$ such that
\[
\mathcal L_{\s_1,\s_2}\subset\operatorname{ell}_h(G),\quad
\operatorname{WF}_h(G)\subset V,\quad
\operatorname{WF}_h(G_0)\subset V\setminus\mathcal L_{\s_1,\s_2},
\]
and, for every $N\geq0$,
\begin{align}\label{eq:radial-estimate-with-cutoffs}
\|Gv\|_{H_h^s}
\lesssim\|G_0v\|_{H_h^s}
+h^{-1}\|G_1P_B(h)v\|_{H_h^{s-1}}
+h^N\|\chi v\|_{H_h^{-N}},
\end{align}
with  constant  independent of $h$ and $v$.
\end{prop}

\begin{proof}
After localizing near $x=\s_1$, we extend the operator to a compact
one-dimensional manifold. Since
$h^2B(h)\in h^2\Psi_h^0$, we have
$\sigma_{{\rm pr},h}(P_B(h))=\sigma_{{\rm pr},h}(P(h))=p$. 
In the boundary chart $\s_2\x>0$, with
$\rho=(\s_2\x)^{-1}$, its rescaled Hamilton vector field is
\[
\rho H_{p}
=
2\s_2(x^2-1)\partial_x
+2\s_2x\rho\partial_\rho.
\]
Consequently, its linearization at
$\mathcal L_{\s_1,\s_2}=\{x=\s_1,\rho=0\}$ is
\[
4\s_1\s_2(x-\s_1)\partial_x
+2\s_1\s_2\rho\partial_\rho,
\]
which agrees with \eqref{eq:radial-linearization}.

We next verify the threshold condition in
\cite[(E.4.47)]{DZbook}. Since the operator has order two, the symbol
appearing there is
\[
\tau_s
:=
\sigma_{{\rm pr},h}\bigl(h^{-1}\operatorname{Im}P_B(h)\bigr)
+
\left(s-\frac12\right)
\frac{H_{p}\langle\x\rangle}{\langle\x\rangle}.
\]
As $P(h)$ is formally self-adjoint,
\[
h^{-1}\operatorname{Im}P_B(h)
=-h\operatorname{Im}B(h)\in h\Psi_h^0.
\]
Hence the first term in $\langle\x\rangle^{-1}\tau_s$ vanishes at
fiber infinity. Moreover,
\[
\langle\x\rangle^{-1}
\frac{H_{p}\langle\x\rangle}{\langle\x\rangle}
=
\langle\x\rangle^{-1}H_{p}\log\langle\x\rangle
=
-\frac{2x\x^3}{\langle\x\rangle^3}
\to -2\s_1\s_2
\]
at $\mathcal L_{\s_1,\s_2}$. Thus, at a radial sink,
i.e.~in the case when $\s_1\s_2=-1$,
\[
\langle\x\rangle^{-1}\tau_s
\to 2\left(s-\textstyle\frac12\right)<0.
\]
This means that the below-threshold condition is satisfied for $s<1/2$, so then
\cite[Thm.~E.54]{DZbook} applies.  
In the radial source case, reversing the Hamilton flow by replacing $P_B(h)$ with $-P_B(h)$ reduces the estimate to the sink case.
\end{proof}

We use Proposition~\ref{prop:basic-radial-estimate} as follows. If
$v=O_{H_h^{-N_0}}(h^{-M_0})$ on a neighborhood containing $\operatorname{supp}\chi$
for some $M_0,N_0\geq0$, and
\[
G_0v=O_{H_h^s}(h^\infty),\quad
G_1P_B(h)v=O_{H_h^{s-1}}(h^\infty),
\]
then $Gv=O_{H_h^s}(h^\infty)$. Indeed, for any $K>0$, taking
$N\geq N_0$ and $N-M_0\geq K$ makes the last term in
\eqref{eq:radial-estimate-with-cutoffs} $O(h^K)$.

\subsection{Conormal regularity and boundary expansion}
We write $c=O_{S^m}(h^\infty)$ if, on the subset of the $\x$-line under consideration,
\[
|\partial_\x^\alpha c(\x,h)|\leq C_{\alpha N}h^N
\langle\x\rangle^{m-\alpha},\quad \alpha,N\in\mathbb N_0.
\]

\begin{prop}[Conormal expansion]\label{prop:conormal-structure}
Fix $\s\in\{\pm\}$. For some $M_0,N_0\geq0$, assume that
$v=O_{H_h^{-N_0}}(h^{-M_0})$ locally near $x=\s$ and that
$P(h)v=O_{\mathcal S}(h^\infty)$. Suppose also that there are neighborhoods
$V_{\s_2}$ of $\mathcal L_{\s,\s_2}$ in
$\overline{T^*\re}$ such that
\[
\operatorname{WF}_h(v)\cap V_{\s_2}\subset\mathcal L_{\s,\s_2},
\quad \s_2\in\{\pm\}.
\]
Then, locally near $x=\s$,
\begin{align}\label{eq:boundary-form-conormal}
v(x)=a^{(0)}_\s(h)\log|x-\s|+a^{(1)}_\s(h)H(x-\s)
+O_{H^{3/2-\delta}}(h^\infty)
\end{align}
for every $\delta>0$, with $a^{(0)}_\s(h),a^{(1)}_\s(h)=O(h^\infty)$.
\end{prop}

In other words, the wavefront set condition states that $v=O_{C^\infty}(h^\infty)$ microlocally
at each point of $V_{\s_2}\setminus\mathcal L_{\s,\s_2}$. In view of $P(h)v=O_{\mathcal S}(h^\infty)$,
it suffices to check this condition on the characteristic set of $p$.
To that end we adapt the iterated module regularity argument of  \cite[Thm.~6.3]{HV}, then verify the
uniform conormal symbol estimates directly by Fourier transform.

\begin{proof}
\proofstep[Test module and its generators]
Fix $\s_2\in\{\pm\}$ and set
\[
N^*_{\s_2}\{x=\s\}
:=\{(x,\x)\in T^*\re\setminus0\mid x=\s,\ \s_2\x>0\}.
\]
Its closure in $\overline{T^*\re}$ meets  fiber infinity at
$\mathcal L_{\s,\s_2}$. Let $U_{\s_2}\subset T^*\re\setminus0$ be an
open conic neighborhood of $N^*_{\s_2}\{x=\s\}$ whose closure, taken in
$T^*\re\setminus0$, is disjoint from $N^*_{-\s_2}\{x=\s\}$.
Here $\Psi^m$ denotes the ordinary properly supported classical
pseudodifferential calculus, and $\sigma_{\rm pr}$ denotes its homogeneous principal symbol. Following \cite[Def.~6.1]{HV}, we introduce the test module
\[
\mathcal M_{\s,\s_2}
:=\{G\in\Psi^1\mid
\sigma_{\rm pr}(G)|_{N^*_{\s_2}\{x=\s\}}=0\}.
\]
Microlocally in $U_{\s_2}$, it is generated over $\Psi^0$, modulo smoothing
operators, by the identity and $A=(x-\s)D_x$:
\begin{align}\label{eq:module-generation}
\mathcal M_{\s,\s_2}=\Psi^0+\Psi^0A.
\end{align}
More precisely, every $G\in\mathcal M_{\s,\s_2}$ can be written
$G=E_1A+E_0+R$ microlocally in $U_{\s_2}$, with $E_0,E_1\in\Psi^0$ and
$R\in\Psi^{-\infty}$. Indeed, by Hadamard's lemma  we can factorize
\[
\sigma_{\rm pr}(G)(x,\x)=(x-\s)\x g_0(x,\x),
\]
where $g_0$ is homogeneous of order zero. If $E_1\in\Psi^0$ satisfies $\sigma_{\rm pr}(E_1)=g_0$, then since 
$\sigma_{\rm pr}(A)=(x-\s)\x$, the operators $G$ and $E_1A$ have the
same principal symbol microlocally in $U_{\s_2}$. Therefore $G-E_1A\in\Psi^0$ microlocally in $U_{\s_2}$.

\proofstep[Iterated regularity]
Choose $\chi\in C^\infty_{\rm c}(\re)$ equal to one near $\s$ and supported
sufficiently close to $\s$, and put $w=\chi v$.
The hypothesis and semiclassical elliptic regularity imply
$v=O_{C^\infty}(h^\infty)$ near $\operatorname{supp}\chi'$ and hence
\[
P(h)w=f_h:=\chi P(h)v+[P(h),\chi]v=O_{C^\infty_{\rm c}}(h^\infty),
\quad
\operatorname{WF}_h(w)\subset
\mathcal L_{\s,+}\cup\mathcal L_{\s,-}.
\]
Such a choice of $\chi$ is possible because the two characteristic curves
of $p$ above $x$ close to $\s$ lie in $V_+\cup V_-$, and $p$ is
elliptic at all remaining points there.

Writing $Q_h(x)=-1-4\pi^2h^2x^2$, we compute
\begin{align}\label{eq:exact-conormal-commutator}
[P(h),A]
=-\frac{2i\s}{x+\s}P(h)+\frac{2h^2\s}{x+\s}A+iR_{\s,h},
\end{align}
where $R_{\s,h}=2\s Q_h/(x+\s)+(x-\s)Q_h'$ is a uniformly bounded
multiplication operator near $x=\s$. Iterating this identity gives, for $k\geq1$,
\begin{align}\label{eq:commuted-module-equation}
(P(h)-h^2B_k)A^k
=D_kP(h)+\sum_{j=0}^{k-1}R_{kj}(h)A^j,
\quad B_k(x)=\frac{2k\s}{x+\s},
\end{align}
locally near $x=\s$. Here $D_k$ is a differential operator of order $k$
and the $R_{kj}(h)$ are uniformly bounded operators of order zero.
 In this computation, the $k$ occurrences of the term
$2h^2\s/(x+\s)A$ in \eqref{eq:exact-conormal-commutator} produce the coefficient
$B_k=2k\s/(x+\s)$ in front of $A^k$. Commuting $A$ past this coefficient
produces only order-zero factors, and the terms involving
$R_{\s,h}$ already contain fewer than $k$ powers of $A$.
Thus all remaining terms involve  indeed $A^j$ with $j<k$.  

Fix $s<1/2$. For $k=0$, Proposition~\ref{prop:basic-radial-estimate} with
$B=0$ implies $w=O_{H_h^s}(h^\infty)$ at both radial points.
Suppose the corresponding estimates hold for $A^jw$, $j<k$.
By \eqref{eq:commuted-module-equation},
\[
(P(h)-h^2B_k)A^kw=O_{H_h^{s-1}}(h^\infty)
\]
microlocally near either radial point. Let $G_1$ be supported where these
finitely many estimates hold. The a priori bound implies
\[
A^kw=O_{H_h^{-N_0-k}}(h^{-M_0-k}),
\]
since $hA$ has semiclassical differential order one. Moreover,
$G_0A^kw=O_{H_h^s}(h^\infty)$, because
$\operatorname{WF}_h(A^kw)\subset\operatorname{WF}_h(w)$ and the
microsupport of $G_0$ avoids the radial points. Applying
Proposition~\ref{prop:basic-radial-estimate} directly to $P(h)-h^2B_k$
completes the induction. Using a microlocal partition of unity argument and applying the
estimates away from the radial points, we obtain
\begin{align}\label{eq:iterated-module-estimate}
A^kw=O_{H_h^{1/2-\varepsilon}(\re)}(h^\infty),
\quad k\in\mathbb N_0,\quad\varepsilon>0.
\end{align}
Since $[A,E]\in\Psi^0$ for $E\in\Psi^0$, by repeated use of
\eqref{eq:module-generation} we can express each finite product of module
operators as a sum of $E_jA^j$, $E_j\in\Psi^0$, modulo smoothing terms.
Thus the corresponding iterated-regularity estimates hold as well.
The smoothing terms are $O_{C^\infty}(h^\infty)$ by
\eqref{eq:iterated-module-estimate}.

\proofstep[Oscillatory integral representation]
We use Melrose's characterization of conormal distributions by iterated
regularity, in the formulation of
\cite[Vol.~IV, Def.~25.1.1 and Prop.~25.1.3]{H}.
For each fixed $h$, by elliptic regularity away from $x=\s$ combined with \eqref{eq:iterated-module-estimate} and
\eqref{eq:module-generation},  $w$ is conormal to
$\{x=\s\}$. Thus it has an oscillatory integral representation with
phase $(x-\s)\eta$.  We want now to infer more precise information on the $h$-dependence.

Put $y=x-\s$, $w_\s(y,h)=w(\s+y,h)$, and let
\[
\widetilde a(\eta,h)=\mathcal F_yw_\s(\eta,h)
\]
be the (ordinary) Fourier transform. Since
\[
\mathcal F_y(yD_yw_\s)
=i(\eta\partial_\eta+1)\widetilde a,
\]
the estimates \eqref{eq:iterated-module-estimate} control arbitrary
powers of $\eta\partial_\eta$ applied to $\widetilde a$.
More precisely, on a dyadic region
$r\leq|\eta|\leq2r$, $r\geq1$, we have
\[
\partial_\eta^\alpha
=\eta^{-\alpha}P_\alpha(\eta\partial_\eta)
\]
for a polynomial $P_\alpha$. After the rescaling $\eta=r\theta$,
using the one-dimensional Sobolev embedding on the fixed region
$1\leq|\theta|\leq2$, for $s\geq 0$, we get therefore 
\[
\sup_{r\leq|\eta|\leq2r}
|\partial_\eta^\alpha\widetilde a(\eta,h)|
\lesssim
r^{-s-\alpha-\frac12}
\sum_{j=0}^{\alpha+1}
\|(yD_y)^jw_\s\|_{H^s}.
\]
Indeed, the $H^s$ weight contributes $r^{-s}$, there is a further $r^{-\alpha}$ factor
due to $\eta^{-\alpha}$, and rescaling the
$L^2$ norm produces a $r^{-1/2}$  factor.

Let $s=1/2-\varepsilon$, with $0<\varepsilon<1/2$. Since
\[
\|(yD_y)^jw_\s\|_{H^s}
\lesssim h^{-s}\|A^jw\|_{H_h^s}
=O(h^\infty),
\]
we obtain
\[
\widetilde a=O_{S^{-1+\varepsilon}}(h^\infty)
\]
for $|\eta|\geq1$. On bounded $\eta$-intervals the same conclusion
follows from the compact support of $w_\s$ and the
$O_{H^s}(h^\infty)$ bound.

We now pass to semiclassical Fourier variables. Fix $R\geq1$ and decrease
$h_0$ so that $2\pi h_0<R/2$. Choose
$\theta_\pm\in C^\infty(\re;[0,1])$ such that
\[
\supp\theta_\pm\subset\{\pm\x\geq R\},
\quad
\theta_\pm=1 \text{ on }\{\pm\x\geq2R\}.
\]
Define
\[
a_{\s_2}(\x,h)
:=\theta_{\s_2}(\x)h^{-1/2}
\widetilde a(\x/h,h).
\]
For $m=-1+\varepsilon$ and $|\x|\geq R$, the amplitude before
multiplication by $\theta_{\s_2}$ satisfies
\[
\begin{split}
	\langle\x\rangle^{-m+\alpha}
	\big|\partial_\x^\alpha(h^{-1/2}\widetilde a(\x/h,h))\big|
	&=
	h^{-1/2-\alpha}
	\langle\x\rangle^{-m+\alpha}
	\big|(\partial_\eta^\alpha\widetilde a)(\x/h,h)\big| \\
	&\lesssim
	h^{-1/2-m}
	\sup_{\eta}
	\langle\eta\rangle^{-m+\alpha}
	|\partial_\eta^\alpha\widetilde a(\eta,h)|.
\end{split}
\]
Since $-1/2-m=1/2-\varepsilon$, the fixed power of $h$ is absorbed
by the $O(h^\infty)$ bound. Applying Leibniz' rule to the fixed cutoff
$\theta_{\s_2}$, we conclude that
\[
a_{\s_2}=O_{S^{-1+\varepsilon}}(h^\infty).
\]

The part not picked up by $\theta_++\theta_-$ is supported in the fixed
region $|\x|\leq 2R$. Its inverse semiclassical Fourier transform is obtained
by applying to $w$ a semiclassical Fourier multiplier with symbol supported
in $|\x|\leq 2R$. Since $w=O_{H_h^s}(h^\infty)$, the resulting function is
$O_{C^\infty}(h^\infty)$: each $x$-derivative introduces at most a fixed
power of $h^{-1}$, which is absorbed by the $O(h^\infty)$ estimate.
Consequently, near $x=\s$,
\[
w=v_++v_-+O_{C^\infty}(h^\infty),
\]
where
\begin{align}\label{eq:homogeneous-conormal-representation}
	v_{\s_2}(x)
	&=\frac{1}{\sqrt{2\pi h}}
	\int e^{\frac{i}{h}(x-\s)\x}
	a_{\s_2}(\x,h)\,d\x,
	\quad
	a_{\s_2}=O_{S^{-1+\varepsilon}}(h^\infty),
\end{align}
and $\supp a_{\s_2}\subset\{\s_2\x\geq R\}$.

\proofstep[Exact phase and transport equation]
Set $q_{\s,\s_2}(\x)=F_{\s,\s_2}(\x)-\s\x$.
The explicit formula \eqref{eq:phase2} implies
\begin{align}\label{eq:conormal-phase-symbol-bounds}
|\partial_\x^j q_{\s,\s_2}(\x,h)|
\leq C_j\langle\x\rangle^{-1-j},
\quad\s_2\x\geq R,\quad j\in\mathbb N_0,
\end{align}
uniformly for $0<h\leq h_0$. In particular,
$|\partial_\x^\alpha e^{iq_{\s,\s_2}/h}|
\lesssim_\alpha h^{-\alpha}\langle\x\rangle^{-\alpha}$.
By Leibniz' rule, multiplication by this factor preserves
$O_{S^m}(h^\infty)$. Hence
\begin{align}\label{eq:exact-phase-conormal-representation}
&v_{\s_2}=J_{\s,\s_2}(b_{\s_2}),\quad
b_{\s_2}=e^{iq_{\s,\s_2}/h}a_{\s_2}
=O_{S^{-1+\varepsilon}}(h^\infty),\nonumber\\
&J_{\s,\s_2}(b)(x)
:=\frac{1}{\sqrt{2\pi h}}\int
 e^{\frac{i}{h}(x\x-F_{\s,\s_2}(\x))}b(\x,h)\,d\x.
\end{align}

We now use the transport equation from Step~1 of the
proof of Proposition~\ref{prop:quasimode}. Put $c=2\pi h$ and
\[
\alpha_{\s,\s_2}(\x)
:=\partial_xp_h(F'_{\s,\s_2}(\x),\x)
=2\s\sqrt{(\x^2+1)(\x^2-c^2)},
\quad
\mathcal L_{H_{p_h}}b:=\alpha_{\s,\s_2}b'
+\tfrac12\alpha'_{\s,\s_2}b.
\]
We have $\mathcal F_hv_{\s_2}=\theta_{\s_2}\mathcal F_hw$ by construction and
$\mathcal F_hf_h=O_{S^{-\infty}}(h^\infty)$ by the bound on $f_h$. Commuting
$\mathcal F_hP(h)\mathcal F_h^{-1}$ with $\theta_{\s_2}$ produces only terms
supported in $R\leq|\x|\leq2R$, where the Fourier amplitude and all its
derivatives are $O(h^\infty)$. Using
$\mathcal F_hv_{\s_2}=e^{-iF_{\s,\s_2}/h}b_{\s_2}$ in the transport identity,
we therefore obtain
\[
ih \mathcal L_{H_{p_h}}b_{\s_2}
-h^2\partial_\x\bigl((\x^2-c^2)\partial_\x b_{\s_2}\bigr)
=O_{S^{-\infty}}(h^\infty).
\]
Since $\partial_\x((\x^2-c^2)\partial_\x)$ preserves each symbol class
$S^m$, we conclude that
\begin{align}\label{eq:transport-symbol-error}
\mathcal L_{H_{p_h}}b_{\s_2}
=O_{S^{-1+\varepsilon}}(h^\infty).
\end{align}

Let $\mathcal G_{\s,\s_2}=|\alpha_{\s,\s_2}|$.
Since $\operatorname{sgn}\alpha_{\s,\s_2}=\s$ on $\s_2\x\geq R$,
\[
\mathcal L_{H_{p_h}}b
=\s\mathcal G_{\s,\s_2}^{1/2}
\partial_\x(\mathcal G_{\s,\s_2}^{1/2}b).
\]
As $\mathcal G_{\s,\s_2}^{-1/2}\in S^{-1}$,
\eqref{eq:transport-symbol-error} implies
\[
\partial_\x(\mathcal G_{\s,\s_2}^{1/2}b_{\s_2})
=O_{S^{-2+\varepsilon}}(h^\infty).
\]
Since $\varepsilon<1$, this is integrable in each of the regions
$\s_2\x\geq R$. Let
$\gamma_{\s_2}(h)$ be the limit of $\mathcal G_{\s,\s_2}^{1/2}b_{\s_2}$ as
$\s_2\x\to\infty$. Evaluating at a fixed point and integrating the derivative
to infinity shows that $\gamma_{\s_2}(h)=O(h^\infty)$. Integrating instead
from $\x$ to infinity proves
\begin{align}\label{eq:transported-conormal-amplitude}
b_{\s_2}(\x,h)
=\gamma_{\s_2}(h)\mathcal G_{\s,\s_2}(\x,h)^{-1/2}
+O_{S^{-2+\varepsilon}}(h^\infty).
\end{align}
Thus the transport equation determines the leading term and improves the
remainder to order $-2+\varepsilon$.

\proofstep[Logarithmic and Heaviside terms]
By \eqref{eq:transported-conormal-amplitude}, after returning from the
exact phase to the homogeneous phase $(x-\s)\x$, the amplitude in the
region $\s_2\x\geq R$ is
\[
\gamma_{\s_2}(h)e^{-iq_{\s,\s_2}/h}
\mathcal G_{\s,\s_2}^{-1/2}
+O_{S^{-2+\varepsilon}}(h^\infty).
\]
At fiber infinity,
\[
\mathcal G_{\s,\s_2}^{-1/2}
-2^{-1/2}|\x|^{-1}\in S^{-3}.
\]
Moreover, by \eqref{eq:conormal-phase-symbol-bounds},
\[
|\partial_\x^\alpha(e^{-iq_{\s,\s_2}/h}-1)|
\lesssim_\alpha
h^{-1-\alpha}\langle\x\rangle^{-1-\alpha},
\quad \s_2\x\geq R.
\]
Since $\gamma_{\s_2}(h)=O(h^\infty)$, all fixed powers of $h^{-1}$
are harmless. Consequently,
\[
\gamma_{\s_2}(h)e^{-iq_{\s,\s_2}/h}
\mathcal G_{\s,\s_2}^{-1/2}
=
2^{-1/2}\gamma_{\s_2}(h)|\x|^{-1}
+O_{S^{-2+\varepsilon}}(h^\infty)
\]
for $\s_2\x\geq R$.

Combining the contributions from $\s_2=+$ and $\s_2=-$, the leading
amplitude can be written as
\[
\theta_\infty(\x)
\bigl(
\gamma^{(0)}_\s(h)|\x|^{-1}
+\gamma^{(1)}_\s(h)\x^{-1}
\bigr),
\quad
\theta_\infty=\theta_++\theta_-,
\]
where
\[
\gamma^{(0)}_\s
=2^{-3/2}(\gamma_++\gamma_-),
\quad
\gamma^{(1)}_\s
=2^{-3/2}(\gamma_+-\gamma_-).
\]
The terms supported in the regions $R\leq|\x|\leq2R$ have amplitudes
which are $O(h^\infty)$ together with all derivatives, and their inverse
semiclassical Fourier transforms are therefore
$O_{C^\infty}(h^\infty)$.

Since $\theta_\infty=1$ for $|\x|\geq2R$, modifying
$\theta_\infty|\x|^{-1}$ or $\theta_\infty\x^{-1}$ at bounded
frequencies changes their inverse Fourier transforms only by a smooth
function. Using a rescaled variant of the identities for the inverse Fourier transforms of $|\x|^{-1}$ and  $\x^{-1}$, we
obtain, for $y=x-\s$,
\[
\mathcal F_h^{-1}(\theta_\infty|\x|^{-1})
=
-\frac{2}{\sqrt{2\pi h}}\log|y|+r_0(y;h),
\]
and
\[
\mathcal F_h^{-1}(\theta_\infty\x^{-1})
=
\frac{i\pi}{\sqrt{2\pi h}}\operatorname{sgn}(y)+r_1(y;h),
\]
where $r_0(\,\cdot\,;h)$ and $r_1(\,\cdot\,;h)$ are smooth near
$y=0$. Their local $C^k$ seminorms have at most polynomial and
logarithmic growth in $h^{-1}$. Since
$\gamma^{(0)}_\s,\gamma^{(1)}_\s=O(h^\infty)$, the corresponding
smooth terms are $O_{C^\infty}(h^\infty)$.

Using $\operatorname{sgn}(y)=2H(y)-1$, the singular part is therefore
\[
a^{(0)}_\s(h)\log|y|+a^{(1)}_\s(h)H(y),
\]
where
\[
a^{(0)}_\s(h)
=
-\frac{2}{\sqrt{2\pi h}}\gamma^{(0)}_\s(h),
\quad
a^{(1)}_\s(h)
=
\frac{2i\pi}{\sqrt{2\pi h}}\gamma^{(1)}_\s(h).
\]
In particular,
\[
a^{(0)}_\s(h),a^{(1)}_\s(h)=O(h^\infty).
\]
The constant term arising from
$\operatorname{sgn}(y)=2H(y)-1$ is smooth and is absorbed into the
$O_{C^\infty}(h^\infty)$ remainder.

It remains to consider the inverse Fourier transform of the
$O_{S^{-2+\varepsilon}}(h^\infty)$ remainder. Let
$r=O_{S^{-2+\varepsilon}}(h^\infty)$. For
$0\leq s<3/2-\varepsilon$, semiclassical Plancherel gives
\[
\|\mathcal F_h^{-1}r\|_{H^s}
\lesssim
h^{-s}\|\langle\x\rangle^s r\|_{L^2_\x}
=O(h^\infty),
\]
since $\langle\x\rangle^{s-2+\varepsilon}\in L^2(\re)$ for 
$s<\frac32-\varepsilon$.
Given $0<\delta<3/2$, we take
$s=3/2-\delta$ and
$0<\varepsilon<\min(\delta,1/2)$. This proves
\eqref{eq:boundary-form-conormal}. The case $\delta\geq3/2$ follows
from the corresponding Sobolev inclusions. 
\end{proof}

\subsection{Proof of completeness}\label{subsection:proof-completeness}
We have now all the ingredients at hand to prove  Proposition \ref{prop:asymexpef}.

\begin{proof}[Proof of Proposition~\ref{prop:asymexpef}]
By Lemmas~\ref{lem:semicenergyest} and \ref{lem:localelliptic}, all derivatives
of $u$ are polynomially bounded in $h^{-1}$ on the two exterior half-lines.
The same holds for the model quasimodes, by the WKB representation
\eqref{eq:quasimodeWKB} and the symbol estimates in
Proposition~\ref{prop:quasimode}$(i)$.
On $[2,\infty)$, put $W_0=W_h(u_{+,+},u_{+,-})$.
Using \eqref{eq:quasimodeWKB}, the conjugation symmetries from the proof of
Proposition~\ref{prop:quasimode}, and the definition of $W_h$, we compute
\[
W_0(3)=|\mathfrak c_{+,+}|^2g(3)
\left(h(a_{+,+}\overline{a_{+,+}'}-a_{+,+}'\overline{a_{+,+}})
-2i\f_{+,+}'|a_{+,+}|^2\right)_{x=3}+O(h^\infty).
\]
At $x=3$, Proposition~\ref{prop:quasimode}$(i)$ implies
$a_{+,+}=a_0+O(h\log(1/h))$, also after one $x$-derivative, where
$a_0=|g\f_{+,+}'|^{-1/2}$. Moreover, \eqref{eq:connformula} implies
$|\mathfrak c_{+,+}|^2=1/2+O(h\log(1/h))$.
Since $g\f_{+,+}'a_0^2=1$ at this point,
\[
W_0(3)=-i+O(h\log(1/h)).
\]
The Wronskian identity \eqref{eq:Wronskian-identity} and the residual estimate
\eqref{eq:quasimodeeq} imply $W_0'=O_{\mathcal S}(h^\infty)$ on
$[2,\infty)$. Consequently,
\[
\sup_{x\geq2}|W_0(x)-W_0(3)|
\leq\int_2^\infty|W_0'(t)|\,dt=O(h^\infty),
\]
so $W_0$ is uniformly bounded away from zero for small $h$.

Define
\[
c_{+,+}(x,h)=\frac{W_h(u,u_{+,-})(x)}{W_0(x)},\quad
c_{+,-}(x,h)=-\frac{W_h(u,u_{+,+})(x)}{W_0(x)}.
\]
By Cramer's rule we get 
$u=c_{+,+}(x,h)u_{+,+}+c_{+,-}(x,h)u_{+,-}$, initially with coefficients
of at most polynomial growth in $x$ and $h^{-1}$.
Differentiating these formulas and using \eqref{eq:Wronskian-identity},
$P(h)u=0$, and \eqref{eq:quasimodeeq}, we obtain
\[
\partial_xc_{+,+},\ \partial_xc_{+,-}
=O_{\mathcal S([2,\infty))}(h^\infty).
\]
For instance,
\[
h\partial_xc_{+,+}
=-\frac{uP(h)u_{+,-}}{W_0}-hc_{+,+}\frac{W_0'}{W_0},
\]
and both terms are Schwartz remainders by the polynomial bounds above.
By integrating from $x=3$ we get that the coefficients are uniformly
polynomially bounded in $h^{-1}$ and have limits $c_{+,\s_2}(h)$ at
infinity. By integrating from $x$ to infinity we  obtain
\begin{align}\label{eq:coefficient-limits}
c_{+,\s_2}(x,h)-c_{+,\s_2}(h)
=O_{\mathcal S([2,\infty))}(h^\infty),\quad\s_2\in\{\pm\}.
\end{align}
Multiplying by the model quasimodes gives
\[
u=c_{+,+}(h)u_{+,+}+c_{+,-}(h)u_{+,-}
+O_{\mathcal S([2,\infty))}(h^\infty).
\]
The analogous argument on $(-\infty,-2]$ allows us to  construct $c_{-,+}(h),c_{-,-}(h)$.
Thus the remainder $v$ in Proposition~\ref{prop:asymexpef} satisfies
\begin{align}\label{eq:exterior-remainder}
P(h)v=O_{\mathcal S(\re)}(h^\infty),\quad
v=O_{\mathcal S(\{|x|\geq2\})}(h^\infty).
\end{align}

Starting from \eqref{eq:exterior-remainder}, we apply propagation of
singularities for operators of real principal type
\cite[Thm.~E.47]{DZbook} along finite bicharacteristic segments of $p$.
On every compact exterior interval, both characteristic curves connect
to the region where the remainder is already controlled.
Together with semiclassical elliptic regularity
\cite[Thm.~E.33]{DZbook}, this proves
$v=O_{C^\infty(K)}(h^\infty)$ for every $K\Subset\re\setminus\{\pm1\}$.
In particular, in sufficiently small fixed neighborhoods of the radial
points, $\operatorname{WF}_h(v)$ contains no points other than the radial
points themselves. Thus, we can apply
Propositions~\ref{prop:basic-radial-estimate} and
\ref{prop:conormal-structure}: the former establishes the
$H_h^{1/2-\delta}$ estimate at the radial points, and the latter gives
\eqref{eq:log-Heaviside-representation}.  Using a microlocal partition of unity,
combined with the exterior Schwartz estimate \eqref{eq:exterior-remainder}, gives
\[
v=O_{H_h^{1/2-\delta}(\re)}(h^\infty),\quad\delta>0.
\]
Finally, Proposition~\ref{prop:quasimode}$(ii)$ and Lemma~\ref{lem:quasiortho}
imply
\[
\bigg\|\sum_{\s_1,\s_2}c_{\s_1,\s_2}u_{\s_1,\s_2}\bigg\|_{L^2}^2
\sim\log(1/h)\sum_{\s_1,\s_2}|c_{\s_1,\s_2}|^2.
\]
Since $v=O_{L^2}(h^\infty)$ and $\|u\|_{L^2}=1$, this proves
\eqref{eq:coefsizeasymp}.
\end{proof}

\section{Bohr--Sommerfeld quantization and logarithmic Weyl law}\label{section:Mainproof}

\begin{thm}\label{thm:BSsemiclassical}
Sufficiently large positive eigenvalues of $P_{\rm sa}$ are simple and can be numbered $E_n$, $n\in\frac12\mathbb Z$, $n\gg1$, so that
\begin{align}\label{eq:approximate-BS}
\mathrm{BS}(E_n^{-1/2})=n+O(E_n^{-\infty}),
\end{align}
where we recall that the function $\mathrm{BS}$ is defined in \eqref{eq:BS-function}.
\end{thm}

\subsection{Boundary equations}\label{subsection:approximate-boundary-system}

We now impose the boundary conditions characterizing the self-adjoint extension $P_{\rm sa}$.

\begin{prop}\label{prop:quasimodeBScondition}
Let $c_{\s_1,\s_2}\in\mathbb C$ be not all zero and suppose that
\begin{align*}
u-\sum_{\s_1,\s_2\in\{\pm\}}c_{\s_1,\s_2}u_{\s_1,\s_2}\in\mathcal S(\re).
\end{align*}
Then $u\in D(P_{\rm sa})$ if and only if $\mathrm{BS}(h)\in\frac12\mathbb Z$ and
\begin{align}\label{eq:semiBS}
(c_{+,+},c_{+,-},c_{-,+},c_{-,-})\in
\begin{cases}
\mathbb C(1,-1,-1,1),&\mathrm{BS}(h)\in\mathbb Z,\\
\mathbb C(1,-1,1,-1),&\mathrm{BS}(h)\in\mathbb Z+\frac12.
\end{cases}
\end{align}
In particular, a non-zero admissible coefficient vector exists if and only if $\mathrm{BS}(h)\in\frac12\mathbb Z$.
\end{prop}

\begin{proof}
Every Schwartz function satisfies the four boundary conditions, so Lemma \ref{lem:quasimodeboundary} shows that $u\in D(P_{\rm sa})$ is equivalent to
\begin{equation}\label{eq:exact-boundary-system}
\begin{aligned}
c_{+,+}+c_{+,-}&=0,\quad
c_{-,+}+c_{-,-}=0,\\
\sum_{\s_1,\s_2}\mathfrak c_{\s_1,\s_2}c_{\s_1,\s_2}&=0,\quad
\sum_{\s_1,\s_2}\s_2\mathfrak c_{\s_1,\s_2}c_{\s_1,\s_2}=0.
\end{aligned}
\end{equation}
Put $A=c_{+,+}$ and $B=c_{-,+}$. From the first line we obtain $c_{+,-}=-A$ and $c_{-,-}=-B$. Writing $\mathfrak c=\mathfrak c_{+,+}$ and using the symmetries in \eqref{eq:connformula}, the second line becomes
\begin{align*}
(\mathfrak c-\overline{\mathfrak c})(A-B)=0,\quad
(\mathfrak c+\overline{\mathfrak c})(A+B)=0.
\end{align*}
A non-zero solution exists precisely when $\mathfrak c\in\mathbb R\cup i\mathbb R$, which, by \eqref{eq:connection-BS}, is equivalent to $\mathrm{BS}(h)\in\frac12\mathbb Z$. If $\mathrm{BS}(h)\in\mathbb Z$, then $\mathfrak c$ is real and the system forces $B=-A$; if $\mathrm{BS}(h)\in\mathbb Z+\frac12$, then $\mathfrak c$ is purely imaginary and we get $B=A$ instead. These are exactly the two lines in \eqref{eq:semiBS}.
\end{proof}

\begin{prop}\label{prop:approximate-boundary-system}
Let $u\in D(P_{\rm sa})$ be a normalized eigenfunction satisfying $P(h)u=0$, and let $c_{\s_1,\s_2}(h)$ be the coefficients given by Proposition \ref{prop:asymexpef}. Then
\begin{equation}\label{eq:approx-boundary-system}
\begin{aligned}
c_{+,+}+c_{+,-}&=O(h^\infty),\quad
c_{-,+}+c_{-,-}=O(h^\infty),\\
\sum_{\s_1,\s_2}\mathfrak c_{\s_1,\s_2}c_{\s_1,\s_2}&=O(h^\infty),\quad
\sum_{\s_1,\s_2}\s_2\mathfrak c_{\s_1,\s_2}c_{\s_1,\s_2}=O(h^\infty).
\end{aligned}
\end{equation}
Consequently, $\mathrm{BS}(h)=n+O(h^\infty)$
for a unique $n\in\frac12\mathbb Z$.
\end{prop}

\begin{proof}
The first two equations follow by comparing the logarithmic boundary coefficients at $x=\pm1$: Lemma \ref{lem:quasimodeboundary} gives those of the model quasimodes, while \eqref{eq:log-Heaviside-representation} shows that the remainder contributes only $O(h^\infty)$. The other two equations follow from the boundary conditions at infinity, since the remainder is $O_{\mathcal S}(h^\infty)$ there.

By \eqref{eq:connformula},
$\mathfrak c_{+,-}=\mathfrak c_{-,+}=\overline{\mathfrak c_{+,+}}$,
$\mathfrak c_{-,-}=\mathfrak c_{+,+}$ and
$\mathfrak c_{+,+}=2^{-1/2}\varrho(h)e^{\pi i\mathrm{BS}(h)}$, where
$\varrho(h)>0$ and $\varrho(h)=1+O(h\log(1/h))$. The first line of
\eqref{eq:approx-boundary-system} implies
$c_{+,-}=-c_{+,+}+O(h^\infty)$ and
$c_{-,-}=-c_{-,+}+O(h^\infty)$. Substitution in the second line gives
\begin{align*}
(\mathfrak c_{+,+}-\overline{\mathfrak c_{+,+}})(c_{+,+}-c_{-,+})&=O(h^\infty),\\
(\mathfrak c_{+,+}+\overline{\mathfrak c_{+,+}})(c_{+,+}+c_{-,+})&=O(h^\infty).
\end{align*}
Multiplying by the complementary factors and adding and subtracting yields
\begin{align*}
(\mathfrak c_{+,+}^2-\overline{\mathfrak c_{+,+}}^{\,2})c_{+,+}&=O(h^\infty),\quad
(\mathfrak c_{+,+}^2-\overline{\mathfrak c_{+,+}}^{\,2})c_{-,+}=O(h^\infty).
\end{align*}
By \eqref{eq:coefsizeasymp} and the first line of \eqref{eq:approx-boundary-system},
$|c_{+,+}|+|c_{-,+}|\gtrsim(\log(1/h))^{-1/2}$. Hence
\begin{align*}
\mathfrak c_{+,+}^2-\overline{\mathfrak c_{+,+}}^{\,2}
=i\varrho(h)^2\sin(2\pi \mathrm{BS}(h))=O(h^\infty).
\end{align*}
The zeros of the sine are simple, so this proves the assertion.
\end{proof}

\subsection{Spectral localization}\label{subsection:spectral-localization}

\begin{lem}\label{lem:quasimodeimpev}
Suppose that there exists $u\in D(P_{\rm sa})$ such that
\begin{align*}
P(h)u=O_{L^2}(h^\infty),\quad \|u\|_{L^2(\re)}\sim1.
\end{align*}
Then there exists $E(h)\in\sigma(P_{\rm sa})$ such that $E(h)-h^{-2}=O(h^\infty)$.
\end{lem}

\begin{proof}
By the spectral theorem,
\begin{align*}
\operatorname{dist}(h^{-2},\sigma(P_{\rm sa}))\|u\|_{L^2}
\leq\|(P_{\rm sa}-h^{-2})u\|_{L^2}
=h^{-2}\|P(h)u\|_{L^2}=O(h^\infty).
\end{align*}
Since the spectrum is closed, one may choose $E(h)\in\sigma(P_{\rm sa})$ realizing this distance.
\end{proof}

\begin{thm}\label{thm:evnearBS}
Recall that $\mathrm{BS}(h)$ is defined in \eqref{eq:BS-function}. If $\mathrm{BS}(h)\in\frac12\mathbb Z$, then, for every $N>0$ and all sufficiently small $h>0$, the interval
\begin{align*}
(h^{-2}-h^N,h^{-2}+h^N)
\end{align*}
contains an eigenvalue of $P_{\rm sa}$.
\end{thm}

\begin{proof}
By Proposition \ref{prop:quasimodeBScondition} there exists a non-zero coefficient vector for which
\begin{align*}
u=\sum_{\s_1,\s_2}c_{\s_1,\s_2}u_{\s_1,\s_2}\in D(P_{\rm sa}).
\end{align*}
By Proposition \ref{prop:quasimode}, $P(h)u=O_{\mathcal S}(h^\infty)$, and Proposition \ref{prop:quasimode}$(ii)$ together with Lemma \ref{lem:quasiortho},
\begin{align*}
\|u\|_{L^2}^2\sim\log(1/h)\sum_{\s_1,\s_2}|c_{\s_1,\s_2}|^2.
\end{align*}
In view of Lemma \ref{lem:quasimodeimpev} this proves the assertion.
\end{proof}

\begin{lem}\label{lem:BSmonotone}
For sufficiently small $h>0$,
\begin{align}\label{eq:BS-derivative}
-\mathrm{BS}'(h)\sim h^{-2}\log(1/h).
\end{align}
In particular, $\mathrm{BS}$ is strictly decreasing. If $\mathrm{BS}(h)=n$ and $\mathrm{BS}(h')=n'$ for distinct $n,n'\in\frac12\mathbb Z$, with $h\sim h'$, then
\begin{align}\label{eq:root-separation}
|h-h'|\gtrsim h^2(\log(1/h))^{-1}.
\end{align}
\end{lem}

\begin{proof}
For $a\geq4$, split
\begin{align*}
I'(a)=\frac12\int_1^\infty
\frac{dx}{\sqrt{x^2-1}\sqrt{x^2+a-1}}
\end{align*}
over $[1,2]$, $[2,\sqrt a]$, and $[\sqrt a,\infty)$. The first and third integrals are $O(a^{-1/2})$, while on $[2,\sqrt a]$ the integrand is $\sim(\sqrt a\,x)^{-1}$. Hence
\begin{align*}
I'(a)\sim a^{-1/2}\log a.
\end{align*}
For $a(h)=1+(4\pi^2h^2)^{-1}$, we have
\begin{align*}
\mathrm{BS}'(h)=2I'(a(h))a'(h)+d'(h).
\end{align*}
Since $a'(h)=-(2\pi^2)^{-1}h^{-3}$ and $d'(h)=O(\log(1/h))$, we conclude \eqref{eq:BS-derivative}. The separation estimate follows from the mean value theorem, the upper bound in \eqref{eq:BS-derivative}, and $|n-n'|\geq1/2$.
\end{proof}

\begin{lem}\label{lem:h-variation-quasimodes}
The model quasimodes can be chosen so that
\begin{align*}
u_{\s_1,\s_2}\in C^1((0,h_0);L^2(\mathbb R)),
\quad \s_1,\s_2\in\{\pm\}.
\end{align*}
Moreover, there are $C,M>0$ such that
\begin{align}\label{eq:model-quasimode-h-variation}
\|\partial_hu_{\s_1,\s_2}(h)\|_{L^2}\leq Ch^{-M},
\quad \s_1,\s_2\in\{\pm\}.
\end{align}
\end{lem}

\begin{proof}
The transport coefficients and their $h$-derivatives satisfy the symbol estimates of Lemmas \ref{lem:transport} and \ref{lem:log-symbol-aj}; the Borel sums can be chosen with the same property. For the representation near $x=\s_1$, write $F_{\s_1,\s_2}(\x)=\s_1\x+q_{\s_1,\s_2}(\x)$ and $\widetilde b=e^{-iq_{\s_1,\s_2}/h}b_{\s_1,\s_2}$. If
\begin{align*}
\mathcal F_h^{-1}a(y)=(2\pi h)^{-1/2}\int e^{iy\x/h}a(\x,h)\,d\x,
\end{align*}
then integration by parts gives
\begin{align*}
\partial_h\mathcal F_h^{-1}a
=\mathcal F_h^{-1}\left(\partial_ha+h^{-1}\partial_\x(\x a)-\frac{1}{2h}a\right).
\end{align*}
By Plancherel, the symbol estimates for $q_{\s_1,\s_2}$ and $b_{\s_1,\s_2}$ therefore imply a polynomial $L^2$ bound.

On the WKB region, direct differentiation gives
\begin{align*}
\partial_h\bigl(\mathfrak c_{\s_1,\s_2}ae^{i\f_{\s_1,\s_2}/h}\bigr)
=e^{i\f_{\s_1,\s_2}/h}\left(\mathfrak c'_{\s_1,\s_2}a
+\mathfrak c_{\s_1,\s_2}\partial_ha
+i\mathfrak c_{\s_1,\s_2}a\frac{h\partial_h\f_{\s_1,\s_2}-\f_{\s_1,\s_2}}{h^2}\right).
\end{align*}
By the formula for $\mathfrak c_{\s_1,\s_2}$,  $|\partial_h\mathfrak c_{\s_1,\s_2}|\leq Ch^{-M_1}$ for some $M_1$, while the  formula for  $\f_{\s_1,\s_2}$ implies
\begin{align*}
|h\partial_h\f_{\s_1,\s_2}(x)-\f_{\s_1,\s_2}(x)|\lesssim L_h(x),
\quad \s_1x\geq2,
\end{align*}
and the estimates in Definition \ref{def:symbolclasses} then imply a polynomial $L^2$ bound for all three terms. The gluing cutoffs are independent of $h$, which proves \eqref{eq:model-quasimode-h-variation}.
\end{proof}

Since $\mathrm{BS}(h)\to\infty$ as $h\to0$ and is strictly decreasing for small $h$, for every sufficiently large $n\in\frac12\mathbb Z$ there is a unique $h_n>0$ such that $\mathrm{BS}(h_n)=n$ by Lemma \ref{lem:BSmonotone}.

\begin{prop}[Uniqueness in a spectral window]\label{prop:rank-one-windows}
There is $K_0>0$ such that, for every fixed $K\geq K_0$ and all sufficiently large $n$, the interval
\begin{align}\label{eq:spectral-window-rank}
(h_n^{-2}-h_n^K,h_n^{-2}+h_n^K)
\end{align}
contains at most one eigenvalue of $P_{\rm sa}$, counted with multiplicity.
\end{prop}

\begin{proof}
We put
\begin{align*}
q_0(h)&=u_{+,+}-u_{+,-}-u_{-,+}+u_{-,-},\\
q_1(h)&=u_{+,+}-u_{+,-}+u_{-,+}-u_{-,-},
\end{align*}
and let $e_j(h)=q_j(h)/\|q_j(h)\|_{L^2}$. From Proposition \ref{prop:quasimode}$(ii)$ and Lemma \ref{lem:quasiortho} we have $\|q_j(h)\|_{L^2}^2\sim\log(1/h)$. By applying Lemma \ref{lem:h-variation-quasimodes} and differentiating  the normalization factor we thus obtain, for some $C,M>0$,
\begin{align}\label{eq:distinguished-quasimode-variation}
\|\partial_he_j(h)\|_{L^2}\leq Ch^{-M}.
\end{align}

Let $E$ be an eigenvalue in \eqref{eq:spectral-window-rank}, let $u$ be a normalized eigenfunction, and set $\widetilde h=E^{-1/2}$. Then $|\widetilde h-h_n|=O(h_n^{K+3})$. By Proposition \ref{prop:approximate-boundary-system} we have
\begin{align*}
\mathrm{BS}(\widetilde h)=m+O(h_n^\infty)
\end{align*}
for some $m\in\frac12\mathbb Z$, while \eqref{eq:BS-derivative} gives
\begin{align*}
{\mathrm{BS}(\widetilde h)-n}=\mathrm{BS}(\widetilde h)-\mathrm{BS}(h_n)=O(h_n^{K+1}\log(1/h_n)).
\end{align*}
Thus $m=n$ when $K$ is large. Using the boundary equations we get that the coefficient vector of $u$ is, modulo $O(h_n^\infty)$, a multiple of $(1,-1,-1,1)$ if $n\in\mathbb Z$, and of $(1,-1,1,-1)$ if $n\in\mathbb Z+\frac12$. Consequently, for $j=0$ in the first case and $j=1$ in the second,
\begin{align*}
\|u-\beta e_j(\widetilde h)\|_{L^2}=O(h_n^\infty),\quad
|\beta|=1+O(h_n^\infty)
\end{align*}
for some $\beta\in\mathbb C$. By \eqref{eq:distinguished-quasimode-variation},
\begin{align*}
\|e_j(\widetilde h)-e_j(h_n)\|_{L^2}=O(h_n^{K+3-M}).
\end{align*}
We choose $K_0>M$. Every normalized eigenfunction whose eigenvalue lies in \eqref{eq:spectral-window-rank} has then $o(1)$ distance  from the same one-dimensional subspace $\mathbb Ce_j(h_n)$. Hence that interval cannot contain two orthogonal normalized eigenfunctions, which proves the proposition.
\end{proof}

\begin{proof}[Proof of Theorem \ref{thm:BSsemiclassical}]
For each $n\gg1$, Theorem \ref{thm:evnearBS} ensures the existence of an eigenvalue in
\begin{align*}
(h_n^{-2}-h_n^K,h_n^{-2}+h_n^K)
\end{align*}
for every $K>0$. By Proposition \ref{prop:rank-one-windows} we know that for $K\geq K_0$  this eigenvalue is unique and simple. Since $K$ can be taken arbitrarily large, 
\begin{align*}
E_n=h_n^{-2}+O(h_n^\infty),
\end{align*}
and hence \eqref{eq:approximate-BS}.

Conversely, let $E=h^{-2}$ be a sufficiently large eigenvalue. Proposition \ref{prop:approximate-boundary-system} implies that $\mathrm{BS}(h)=n+O(h^\infty)$ for a unique $n\in\frac12\mathbb Z$. By \eqref{eq:BS-derivative}, $h-h_n=O(h^\infty)$. Fixing $K\geq K_0$, we obtain
\begin{align*}
E\in(h_n^{-2}-h_n^K,h_n^{-2}+h_n^K).
\end{align*}
By Proposition \ref{prop:rank-one-windows}, this interval contains exactly the eigenvalue $E_n$, counted with multiplicity. Hence $E=E_n$, so every sufficiently large positive eigenvalue occurs in the sequence constructed above.
\end{proof}

\subsection{Proof of main result and logarithmic Weyl law}\label{ss:finalproof}
We have now all the ingredients needed to conclude the main result, Theorem \ref{thm:BSthm}, and Corollary \ref{cor:Weyllaw}.

\begin{proof}[Proof of Theorem \ref{thm:BSthm}]
We set $R(E)=d(E^{-1/2})$. By Lemma \ref{lem:first-log-correction-d} we have the asymptotics
\begin{align*}
R(E)=-\frac{1}{16\pi}E^{-1/2}\log E+O(E^{-1/2}),
\end{align*}
and the derivative bounds in \eqref{eq:def-d} imply
\begin{align*}
R'(E)=-\frac12E^{-3/2}d'(E^{-1/2})=O(E^{-3/2}\log E).
\end{align*}
 The asserted numbering of the eigenvalues and the quantization rule \eqref{BS} follow now from Theorem \ref{thm:BSsemiclassical}.
\end{proof}

\begin{proof}[Proof of Corollary \ref{cor:Weyllaw}]
We recall from \cite[(25)]{CM} that
\begin{align*}
I(a)=\frac12\sqrt a(\log a-2+4\log2)+1+o(1),\quad a\to\infty.
\end{align*}
The left hand side of \eqref{BS}, with $E_n$ replaced by a real variable $E$, is strictly increasing for $E$ large by Lemma \ref{lem:BSmonotone}. The finitely many positive eigenvalues below the range of Theorem \ref{thm:BSsemiclassical} are absorbed into the $O(1)$ remainder. Counting the half-integers in \eqref{BS} therefore gives
\begin{align*}
N_+(E)
&=2\biggl(2I\biggl(1+\frac{E}{4\pi^2}\biggr)+\frac14+R(E)\biggr)+O(1)\\
&=2\sqrt{1+\frac{E}{4\pi^2}}
\biggl(\log\biggl(1+\frac{E}{4\pi^2}\biggr)-2+4\log2\biggr)+O(1)\\
&=4\frac{\sqrt E}{2\pi}
\biggl(\log\frac{\sqrt E}{2\pi}-1+2\log2\biggr)+O(1).
\end{align*}
Here we used $R(E)=o(1)$ and, with $t=E/(4\pi^2)$,
\begin{align*}
\sqrt{t+1}-\sqrt t=o(1),\quad
\sqrt{t+1}\log(t+1)-\sqrt t\log t=o(1).
\end{align*}
\end{proof}

\appendix

\section{\texorpdfstring{Self-adjoint extensions of $P$}{Self-adjoint extensions of P}} \label{app:saP}

\subsection{Generalities on self-adjoint extensions} 
We  briefly recall the von Neumann theory of self-adjoint extensions of unbounded operators and then reformulate it in the language of boundary forms used in \cite[Sec.~6]{Katsnelson} and \cite[Eq.~(6)]{CM}. This will be useful for us in \S\ref{ss:ap} where we review the results in \cite{CM} on self-adjoint  extensions of $P$.

\medskip

    Let $T:D(T)\to H$ be a densely defined closed symmetric operator on a Hilbert space $H$. Let us recall that   the {deficiency subspaces} are by definition
    \[ K_\pm := \ker(T^* \mp i) = \operatorname{ran}(T \pm i)^\perp \]
    and the {deficiency indices} are $n_\pm := \dim K_\pm$.

    If $T$ is a closed symmetric operator, then it is well-known that $T$ admits a self-adjoint extension if and only if $n_+ = n_-$.
    Moreover, $T$ is self-adjoint if and only if $n_+ = n_- = 0$.

\begin{thm} (see e.g.~\cite[Thm.~13.10]{Schmudgen})\\
    Let $T$ be a closed symmetric operator with deficiency indices $n_\pm = n$. There is a bijection between
    \begin{itemize}
        \item[(i)] self-adjoint extensions $T_U$ of $T$, and
        \item[(ii)] unitary operators $U: K_+ \to K_-$,
    \end{itemize}
    given by
    \[
      D(T_U) = D(T)+(I+U)K_+,
    \]
    where one sets $T_U = T^*|_{D(T_U)}$.
\end{thm}
We will now reformulate this standard result in terms of Lagrangian or self-orthogonal subspaces of an intrinsic quotient space $\mathcal E$, referred to as the boundary space by Katsnelson in \cite{Katsnelson}.

\begin{defn}\label{def:boundary-form}
Let
\[
    \widetilde\Omega(u,v)
    :=\frac{1}{i}\big(
        \langle T^*u,v\rangle
        -\langle u,T^*v\rangle
    \big),
    \quad u,v\in D(T^*).
\]
Since $\widetilde\Omega$ vanishes whenever either argument belongs to
$D(T)$, it descends to a Hermitian sesquilinear form on the boundary space \(\mathcal E := D(T^*)/D(T)\): 
\[
    \Omega([u],[v])
    :=\widetilde\Omega(u,v).
\]
We call $\Omega$ the boundary form.
\end{defn}

\begin{rem}
Since $T$ is closed, $T^{**}=T$. Consequently, the boundary form
$\Omega$ is nondegenerate on $\mathcal E$.
\end{rem}

\begin{thm}\label{thm:lagrangian-selfadjoint-extensions}
  There is a one-to-one correspondence between the self-adjoint extensions $T_U$ of $T$ and the self-orthogonal subspaces $\mathcal L_\mathcal E \subset \mathcal E$. Specifically, $D(T_U)$ is the pre-image of $\mathcal L_\mathcal E$ under the canonical quotient map $\pi : D(T^*) \to \mathcal E$.
\end{thm}

\begin{proof}[Sketch of proof]
For a subspace $\mathcal L\subset\mathcal E$, set
\[
    T_{\mathcal L}
    :=T^*|_{\pi^{-1}(\mathcal L)}.
\]
Then $T\subset T_{\mathcal L}\subset T^*$. We claim that
\(
    T_{\mathcal L}^*
    =T_{\mathcal L^{\perp_\Omega}}.
\)
Indeed, $y\in D(T_{\mathcal L}^*)$ if and only if
\(
    \langle T^*x,y\rangle
    =\langle x,T^*y\rangle
\) for all $ x\in\pi^{-1}(\mathcal L)$,
where the inclusion $D(T)\subset D(T_{\mathcal L})$ first implies that
$y\in D(T^*)$. By definition of the boundary form, this is equivalent to
\(
    \Omega([x],[y])=0
\) for all  $[x]\in\mathcal L$ or equivalently $[y]\in\mathcal L^{\perp_\Omega}$. Hence
$T_{\mathcal L}^*=T_{\mathcal L^{\perp_\Omega}}$. It follows that $T_{\mathcal L}$ is self-adjoint if and only if
$\mathcal L=\mathcal L^{\perp_\Omega}$. Conversely, every self-adjoint
extension $\widetilde T$ of $T$ satisfies
\[
    T\subset\widetilde T=\widetilde T^*\subset T^*,
\]
and is therefore obtained by taking
$\mathcal L=\pi(D(\widetilde T))$.
\end{proof}

\subsection{\texorpdfstring{Application to $P$}{Application to P}} \label{ss:ap}
For the reader's convenience, we recall the arguments made by Connes and Moscovici in  
\cite{CM}. Let $P_{\min}$ be the closure of $P|_{\mathcal S(\mathbb R)}$ and
$P_{\max}=P_{\min}^*$. Its domain is
\[
    D_{\max}(P)=\{u\in L^2(\mathbb R) \mid Pu\in L^2(\mathbb R)\},
\]
where $Pu$ is understood distributionally. We write $g(x)=x^2-1$ and introduce the
weighted Wronskian
\[
    W_g(u,v)(x)
    :=-g(x)\big(u(x)v'(x)-v(x)u'(x)\big)
\]
(note that in the main part of the paper we use instead the semiclassical weighted Wronskian: $W_g = -h^{-1}W_h$).
For $u\in D(P_{\max})$, the identity
\[
    (gu')'=-Pu-4\pi^2x^2u\in L^2_{\mathrm{loc}}
\]
implies that $gu'\in H^1_{\mathrm{loc}}$ and hence $gu'$ has a continuous representative, also across $x=\pm1$.

By \eqref{eq:max-domain-boundary-form}, 
$u(x)-b_\sigma^{(0)}\log|x-\sigma|$ has finite one-sided
limits at $\sigma=\pm1$: the remaining terms are a step
function and an $H^1_{\mathrm{loc}}$ function.
In particular, $g(x)u(x)\to0$ as $x\to\sigma$.\\
Recall from Definition~\ref{def:boundary-form} the definition of $\widetilde\Omega$. Integration by parts gives
\begin{equation}\label{eq:weighted_wronskian}
    i\widetilde\Omega(u,v)
    =
    -W_g(u,\overline v)\big|_{-\infty}^{-1}
    -W_g(u,\overline v)\big|_{-1}^{1}
    -W_g(u,\overline v)\big|_{1}^{\infty},
    \quad u,v\in D(P_{\max}),
\end{equation}
where the limits at $\pm1$ are taken from \textit{within} each interval (cf.~\cite[Eqs.~(7)--(10)]{CM}, noting that the operator considered there is $-P$).
Passing to the quotient, we obtain the
boundary form $\Omega$ on
$\mathcal E=D(P_{\max})/D(P_{\min})$.

\begin{lem}\cite[Lem.~1.1]{CM}
    The deficiency indices of $P_{\min}$ are $(4,4)$.
\end{lem}

\begin{proof}

Near $x=\pm1$, a local basis of solutions of $(P-z)u=0$ consists of a regular and a logarithmic solution; at either infinity a basis has leading terms $|x|^{-1}e^{\pm2\pi ix}$.
Thus all six endpoints are of limit-circle type, cf.~\cite[proof of Lemma~1.1]{CM}.
A solution of $(P-z)u=0$, $z\in\mathbb C\setminus\mathbb R$, is therefore
described by six parameters on the three components of
$\mathbb R\setminus\{\pm1\}$ and is square-integrable at every endpoint.
The continuity of $gu'$ requires the coefficients of the
logarithmic terms on the two sides of each singular point to agree.
These two relations are independent, since the logarithmic coefficient on each exterior interval can be chosen freely.
Conversely, matching these coefficients makes $gu'$ continuous, so no delta terms arise in $(gu')'$. Finite jumps in $u$ impose no further condition, since $g\delta_{\pm1}=0$.
The resulting piecewise solutions therefore belong to $D(P_{\max})$ and solve the equation distributionally on $\mathbb R$. Thus
$\dim\ker(P_{\max}-z)=4$.
\end{proof}

Since $P$ commutes with parity, one has the $\Omega$-orthogonal decomposition
$\mathcal E=\mathcal E_+\oplus\mathcal E_-$, and the deficiency indices computation gives $\dim\mathcal E=8$. A simple calculation shows that $P$ commutes with $\mathcal{F}$ on $\mathcal S(\mathbb R)$. Since $\mathcal F$ preserves $\mathcal S(\mathbb R)$, it also preserves $D(P_{\min})$ and $D(P_{\max})$, commutes with the corresponding operators, and preserves $\Omega$.

Let us choose a real even $\chi\in C_\mathrm c^\infty(\mathbb R)$ equal to $1$ near $x=\pm1$, put
$\alpha_+=\chi\log|1-x^2|$, and set
$\beta_+=\mathbf 1_{[-1,1]}$, $\alpha_-=x\alpha_+$ and
$\beta_-=x\beta_+$. These functions belong to $D(P_{\max})$.
We further write
$\widehat\alpha_\pm=\mathcal F\alpha_\pm$ and
$\widehat\beta_\pm=\mathcal F\beta_\pm$, where $\mathcal{F}$ is the unnormalized Fourier transform.

\begin{lem}\cite[Lem.~1.5]{CM}
    The equivalence classes in $\mathcal{E}$
    \[
        [\alpha_\pm],\ [\beta_\pm],\
        [\widehat\alpha_\pm],\ [\widehat\beta_\pm]
    \]
    form a basis of $\mathcal E_\pm$.
\end{lem}

\begin{proof}
    Near $\pm1$ one has $g\alpha_+'=2x$, so \eqref{eq:weighted_wronskian} gives
    $\Omega([\alpha_+],[\beta_+])=4i$. The same holds in the odd sector, since
    $W_g(\alpha_-,\beta_-)=x^2W_g(\alpha_+,\beta_+)$.
    Pairings between the original representatives and their Fourier transforms vanish, since the former are compactly supported, while the latter are smooth across $\pm1$, where the one-sided Wronskian limits thus cancel.
    Using the fact that $\Omega$ is invariant under Fourier transform we conclude that it is represented by the matrix
    \begin{equation}\label{eq:bfm}
        \begin{pmatrix}
            0&4i&0&0\\
            -4i&0&0&0\\
            0&0&0&4i\\
            0&0&-4i&0
        \end{pmatrix}.
    \end{equation}
    This matrix is invertible, so the four classes in each sector are linearly independent. Since $\dim\mathcal E=8$, they form bases of $\mathcal E_\pm$, and $\dim\mathcal E_\pm=4$.
\end{proof}

For $[\gamma]\in\mathcal E$, define
$\ell_\gamma([u]):=i\Omega([u],[\gamma])$. In identifying the conditions below, we suppress nonzero scalar factors, since only the kernels matter in the analysis.

\begin{lem}\label{lem:boundary-functionals}\cite[Eqs.~(13)--(14) and (17)--(19)]{CM}
Let $u=u_{\rm even}+u_{\rm odd}\in D(P_{\max})$, with $u_{\rm even/odd}$ the even/odd part of $u$. Then the two conditions
$\ell_{\beta_\pm}([u_\pm])=0$ together are equivalent to \eqref{eq:bc-singular}, while
$\ell_{\widehat\beta_+}([u_{\rm even}])=0$ and $\ell_{\widehat\beta_-}([u_{\rm odd}])=0$ are equivalent to \eqref{eq:bc-infty-even} and \eqref{eq:bc-infty-odd}, respectively:
\begin{align}
    &\lim_{x\to\pm1}(1-x^2)u'(x)=0, \label{eq:bc-singular}\\
    &\lim_{x\to\pm\infty}
    \Big(
        x\sin(2\pi x)u_{\rm even}'(x)
        -(2\pi x\cos(2\pi x)-\sin(2\pi x))u_{\rm even}(x)
    \Big)=0, \label{eq:bc-infty-even}\\
    &\lim_{x\to\pm\infty}
    \Big(
        x\cos(2\pi x)u_{\rm odd}'(x)
        +(2\pi x\sin(2\pi x)+\cos(2\pi x))u_{\rm odd}(x)
    \Big)=0. \label{eq:bc-infty-odd}
\end{align}
\end{lem}

\begin{proof}
We first consider the singular points $\pm1$. For
$u_\text{even}\in D(P_{\max}^+)$, the function $\beta_+$ equals $1$ on
$(-1,1)$ and $0$ outside. Hence using \eqref{eq:weighted_wronskian} we get
\[
    i\widetilde\Omega(u_\text{even},\beta_+)
    =
    -\lim_{x\to 1^-}g(x)u_\text{even}'(x)
    +\lim_{x\to -1^+}g(x)u_\text{even}'(x).
\]
By parity the two limits are opposite, and therefore
$\ell_{\beta_+}([u_\text{even}])=0$ if and only if
$\lim_{x\to 1^-}(1-x)u_\text{even}'(x)=0$. For
$u_\text{odd}\in D(P_{\max}^-)$ one similarly has, on $(-1,1)$,
\[
    W_g(u_\text{odd},\beta_-)
    =g(x)\big(xu_\text{odd}'(x)-u_\text{odd}(x)\big).
\]
Since $g(x)u_\text{odd}(x)\to0$ as $x\to\pm1$, parity again implies
$\ell_{\beta_-}([u_\text{odd}])=0$ if and only if
$\lim_{x\to 1^-}(1-x)u_\text{odd}'(x)=0$. By the continuity of
$gu'$ across $\pm1$, these two
conditions together are equivalent to \eqref{eq:bc-singular}.

At infinity, we have by \cite[Lemma~1.4(ii)]{CM}
\[
    \widehat\beta_+(x)=\frac{\sin(2\pi x)}{\pi x}.
\]
A direct computation yields
\[
    W_g(u_\text{even},\widehat\beta_+)(x)
    =
    \frac{x^2-1}{\pi x^2}
    \Big(
        x\sin(2\pi x)u_\text{even}'(x)
        -(2\pi x\cos(2\pi x)-\sin(2\pi x))u_\text{even}(x)
    \Big).
\]
Since $\widehat\beta_+$ is smooth, $gu_\text{even}'$ is continuous and $gu_\text{even}\to0$ at $\pm1$,  the contributions from the finite endpoints  in \eqref{eq:weighted_wronskian} cancel.
Since the Wronskian is odd, the contributions from $+\infty$ and
$-\infty$ add rather than cancel in the boundary form. Since $(x^2-1)/x^2\to1$, we obtain
\eqref{eq:bc-infty-even}.

Finally,
\[
    \widehat\beta_-(x)
    =\mathcal F(x\beta_+)(x)
    =\frac{i}{2\pi}\widehat\beta_+'(x)
    =\frac{i}{2\pi^2x^2}
      \big(2\pi x\cos(2\pi x)-\sin(2\pi x)\big).
\]
To justify discarding the lower-order terms, we note that $u,u'=O(|x|^{-1})$ by the following argument.
For $f=Pu\in L^2$, choose on each exterior half-line
a basis $u_1,u_2$ of $Pu=0$ with
$u_j,u_j'=O(|x|^{-1})$.
By variation of constants we obtain
\[
    u=c_1u_1+c_2u_2,\quad
    u'=c_1u_1'+c_2u_2',\quad
    |c_j'(x)|\lesssim \frac{|f(x)|}{|x|},
\]
since the weighted Wronskian is constant and nonzero.
The last bound is integrable by Cauchy--Schwarz, so the
coefficients have finite limits at infinity. Consequently,
$u,u'=O(|x|^{-1})$, and the omitted terms tend to zero.

Substituting the  expression for  $\widehat\beta_-$   gives that $W_g(u_\text{odd},\overline{\widehat\beta_-})$ equals
\[
-\frac{i(x^2-1)}{\pi x^2}
    \Big(
        x\cos(2\pi x)u_\text{odd}'(x)
        +(2\pi x\sin(2\pi x)+\cos(2\pi x))u_\text{odd}(x)
        +o(1)
    \Big)
\]
as $x\to\infty$; the omitted terms are $O(|u_\text{odd}(x)|+|u_\text{odd}'(x)|)$.
The finite-endpoint contributions cancel as before, and parity yields the same condition at $-\infty$. This proves \eqref{eq:bc-infty-odd}.
\end{proof}

\begin{defn}
    We denote by $P_{\mathrm{sa}}$ the self-adjoint extension of
    $P_{\min}$ corresponding to $\mathcal L_\beta$, where
    \[
    \mathcal L_\beta
    :=
    \bigcap_{\pm}\ker\ell_{\beta_\pm}
    \cap
    \bigcap_{\pm}\ker\ell_{\widehat\beta_\pm} = \operatorname{span}
    \{[\beta_+],[\beta_-],
      [\widehat\beta_+],[\widehat\beta_-]\} \subset \mathcal E.
    \]
\end{defn}

By \eqref{eq:bfm} applied in each parity sector, the four generators of $\mathcal L_\beta$ are pairwise $\Omega$-orthogonal, so $\mathcal L_\beta$ is isotropic. Since $\mathcal L_\beta$ is four-dimensional, while
$\dim\mathcal E=8$ and $\Omega$ is nondegenerate,
$\mathcal L_\beta$ is maximal isotropic, hence Lagrangian.

\begin{rem}
    By Lemma~\ref{lem:boundary-functionals} and Theorem~\ref{thm:lagrangian-selfadjoint-extensions}, $D(P_{\mathrm{sa}})$ is exactly the
set of $u=u_\text{even}+u_\text{odd}\in D(P_{\max})$ satisfying
\eqref{eq:bc-singular}--\eqref{eq:bc-infty-odd}.
\end{rem}

\begin{rem}
Near $x=\pm1$ an element of $D(P_{\max})$ may have a logarithmic term 
with coefficient  proportional to
$\lim_{x\to\pm1}(1-x^2)u'(x)$. Thus
\eqref{eq:bc-singular} eliminates the logarithmic term, but allows finite jumps at $\pm1$.
For eigenfunctions, the conditions
\eqref{eq:bc-infty-even} and \eqref{eq:bc-infty-odd} select respectively
the asymptotic profiles $\frac{\sin(2\pi x)}{x}$ and $\frac{\cos(2\pi x)}{x}$ in the
even and odd sectors.
\end{rem}

Let $Q$ denote multiplication by $\mathbf 1_{[-1,1]}$ and set
$\widehat Q=\mathcal FQ\mathcal F^{-1}$.

\begin{thm}[Connes--Moscovici {\cite[Thm.~1.6]{CM}}]
\label{thm:CM-extension}
    The operator $P_{\mathrm{sa}}$ commutes with
    $\mathcal F$, $Q$ and $\widehat Q$. Moreover, it is the unique
    self-adjoint extension of $P_{\min}$ commuting with $Q$ and
    $\widehat Q$.
\end{thm}

\begin{proof}[Sketch of proof]Since \(Q\) applied to smooth functions which agree with \(1\) and \(x\) on \([-1,1]\) produces \(\beta_+\) and \(\beta_-\) by \textit{definition} of $\beta_\pm$, any self-adjoint extension commuting with \(Q\) must contain these two boundary classes. Likewise, commutation with \(\widehat Q=\mathcal FQ\mathcal F^{-1}\) forces the classes \([\widehat\beta_+]\) and \([\widehat\beta_-]\). These four boundary classes span $\mathcal L_\beta$. Hence the boundary
	subspace of any self-adjoint extension commuting with both $Q$ and
	$\widehat Q$ contains $\mathcal L_\beta$. Since both subspaces are
	Lagrangian, they coincide, which proves uniqueness. Finally, \(\mathcal L_\beta\) is easily seen to be invariant under \(\mathcal F\), so the resulting extension also commutes with \(\mathcal F\).
\end{proof}

\section{Symbolic estimates}\label{appendix:symbolic-estimates}
The purpose of this appendix is to justify the two symbol estimates used in
\S\ref{subsection:phases-model-quasimodes}. Lemma~\ref{lem:transport} proves
that the Fourier transport coefficients defined in
\eqref{eq:b_0}--\eqref{eq:b_j} satisfy $b_j\in S^{-1-j}$, together with
the corresponding $h$-derivative estimates. Lemma~\ref{lem:log-symbol-aj}
proves that the WKB coefficients defined in
\eqref{eq:WKBa_0}--\eqref{eq:WKBa_j} satisfy
$a_j\in\mathcal S^j_{-1/2-j}$ in the sense of
Definition~\ref{def:symbolclasses}. These estimates are used in
Lemmas~\ref{lem:first-log-correction-d} and \ref{lem:h-variation-quasimodes}.

\subsection{Fourier transport coefficients}
\begin{lem}[Symbol class of transport coefficients]\label{lem:transport}
Let $0<\x_0<\x_1$ and $h_0>0$ satisfy $2\pi h_0<\x_0/2$, and let
$0<h\leq h_0$. Choose $\chi\in C^\infty(\re)$ with $\chi=0$ for
$\x\leq\x_0$ and $\chi=1$ for $\x\geq\x_1$. Put $c=2\pi h$ and
\[
\alpha(\x,h)=\partial_xp_h(F'(\x),\x)
=2\sqrt{(\x^2+1)(\x^2-c^2)},\quad
b_0=\chi\alpha^{-1/2}.
\]
For $j\geq0$, define
\[
r_j=-\partial_\x\bigl((\x^2-c^2)\partial_\x b_j\bigr),\quad
b_{j+1}(\x,h)=i\chi(\x)\alpha(\x,h)^{-1/2}
\int_\infty^\x\alpha(\eta,h)^{-1/2}r_j(\eta,h)\,d\eta.
\]
Then $\partial_h^\beta b_j\in S^{-1-j}$ for all $j,\beta\in\mathbb N_0$.
\end{lem}

\begin{proof}
We work on the $(+,+)$ component. Since $c^2\x^{-2}\leq1/4$ for
$\x\geq\x_0$,
\[
\alpha^{-1/2}
=2^{-1/2}\x^{-1}(1+\x^{-2})^{-1/4}(1-c^2\x^{-2})^{-1/4}.
\]
Differentiating this expression proves
\begin{align}\label{eq:bound on A(xi,h)}
|\partial_\x^\alpha\partial_h^\beta\alpha^{-1/2}(\x,h)|
\leq C_{\alpha\beta}\langle\x\rangle^{-1-\alpha},
\quad\x\geq\x_0,
\end{align}
for all $\alpha,\beta\in\mathbb N_0$, and hence the assertion for $b_0$.
Assume it holds for $b_j$. The exact transport identity in Step~1 of the
proof of Proposition~\ref{prop:quasimode} identifies the remainder as
\[
r_j=-(\x^2-c^2)b_j''-2\x b_j'.
\]
Its coefficients have orders two and one, respectively, compensated by
the corresponding derivatives of $b_j$. Thus
$\partial_h^\beta r_j\in S^{-1-j}$ for every $\beta$.
Set
\[
u_j=\alpha^{-1/2}r_j,\quad
I_j(\x,h)=\int_\x^\infty u_j(\eta,h)\,d\eta.
\]
By \eqref{eq:bound on A(xi,h)} and Leibniz' rule,
\[
|\partial_\eta^\alpha\partial_h^\beta u_j(\eta,h)|
\leq C_{\alpha\beta j}\langle\eta\rangle^{-2-j-\alpha},\quad
|\partial_\x^\alpha\partial_h^\beta I_j(\x,h)|
\leq C_{\alpha\beta j}\langle\x\rangle^{-1-j-\alpha}.
\]
For $\alpha=0$ the latter estimate follows by integration; for
$\alpha\geq1$ use $\partial_\x^\alpha I_j=-\partial_\x^{\alpha-1}u_j$.
Since $b_{j+1}=-i\chi\alpha^{-1/2}I_j$, another application of
Leibniz' rule proves $\partial_h^\beta b_{j+1}\in S^{-2-j}$.
\end{proof}

\subsection{WKB transport coefficients}
\begin{lem}[Logarithmic symbol class of $a_j$]\label{lem:log-symbol-aj}
On the $(+,+)$ component and for $0<h\leq1$, let $a_0=|g\f'|^{-1/2}$
and, for $j\geq0$,
\[
a_{j+1}(x,h)=\frac{i}{2}|g(x)\f'(x)|^{-1/2}
\int_\infty^x|g(y)\f'(y)|^{-1/2}
\bigl(g(y)a_j''(y,h)+g'(y)a_j'(y,h)\bigr)\,dy.
\]
Then $a_j\in\mathcal S^j_{-1/2-j}$ for $x\geq2$. The corresponding
estimates on the other components follow by conjugation and reflection.
\end{lem}

\begin{proof}
For $x\geq2$, put $X=1+x$, $Y=1+hx$ and $L_h=1+\log(Y/(hX))$.
We have
\[
a_0(x,h)=(x^2-1)^{-1/4}f(hx),\quad
f(t)=(1+4\pi^2t^2)^{-1/4}.
\]
The estimates $|f^{(r)}(t)|\leq C_r(1+t)^{-1/2-r}$ and
$\partial_h^\beta f(hx)=h^{-\beta}(hx)^\beta f^{(\beta)}(hx)$,
together with the product rule, imply
\begin{align}\label{eq:rho-bound}
|\partial_x^\alpha\partial_h^\beta a_0(x,h)|
\leq C_{\alpha\beta}h^{-\beta}X^{-1/2-\alpha}Y^{-1/2}.
\end{align}
Thus $a_0\in\mathcal S^0_{-1/2}$.

Suppose inductively that $a_j\in\mathcal S^j_{-1/2-j}$, so
\begin{align}\label{eq:aj-induction-bound}
|\partial_x^\alpha\partial_h^\beta a_j(x,h)|
\leq C_{\alpha\beta j}h^{-\beta}X^{-1/2-\alpha}
Y^{-1/2-j}L_h(x)^j.
\end{align}
In particular,
\begin{align}\label{eq:gj-bound}
|g\partial_h^\beta a_j''+g'\partial_h^\beta a_j'|
\leq C_{\beta j}h^{-\beta}X^{-1/2}Y^{-1/2-j}L_h(x)^j.
\end{align}
Define $F_j=a_0(ga_j''+g'a_j')$.
Combining \eqref{eq:rho-bound}--\eqref{eq:gj-bound} with
$|\partial_x^r g|\leq C_rX^{2-r}$ and
$|\partial_x^r g'|\leq C_rX^{1-r}$, we obtain
\begin{align}\label{eq:Fj-bound}
|\partial_x^\alpha\partial_h^\beta F_j(x,h)|
\leq C_{\alpha\beta j}h^{-\beta}X^{-1-\alpha}
Y^{-1-j}L_h(x)^j.
\end{align}

For $m\geq1$ and $\ell\geq0$, the integral estimate needed below is
\begin{align}\label{eq:log-integral-bound}
\int_x^\infty(1+y)^{-1}(1+hy)^{-m}L_h(y)^\ell\,dy
\leq(1+hx)^{-m}L_h(x)^{\ell+1}.
\end{align}
Indeed, $L_h'=-(1-h)/(XY)$ and
\[
-\partial_x\bigl(Y^{-m}L_h^{\ell+1}\bigr)
=X^{-1}Y^{-m}L_h^\ell
\left(\frac{mhXL_h+(\ell+1)(1-h)}{Y}\right)
\geq X^{-1}Y^{-m}L_h^\ell.
\]
Here we used $m,L_h,\ell+1\geq1$ and $hX+1-h=Y$.
Integrating to infinity proves \eqref{eq:log-integral-bound}, since
$Y^{-m}L_h^{\ell+1}\to0$ there for fixed $h>0$.

Set $I_j(x,h)=\int_x^\infty F_j(y,h)\,dy$.
Applying \eqref{eq:log-integral-bound} with $m=j+1$, $\ell=j$ to
\eqref{eq:Fj-bound} gives
\begin{align}\label{eq:Ij-zero-bound}
|\partial_h^\beta I_j(x,h)|
\leq C_{\beta j}h^{-\beta}Y^{-1-j}L_h(x)^{j+1}.
\end{align}
Differentiation under the integral is justified by
\eqref{eq:Fj-bound}, locally uniformly for $h>0$.
For $\alpha\geq1$, use
$\partial_x^\alpha\partial_h^\beta I_j
=-\partial_x^{\alpha-1}\partial_h^\beta F_j$ and \eqref{eq:Fj-bound}.
Combining these bounds with \eqref{eq:Ij-zero-bound} and $L_h\geq1$, we obtain
\begin{align}\label{eq:Ij-bound}
|\partial_x^\alpha\partial_h^\beta I_j(x,h)|
\leq C_{\alpha\beta j}h^{-\beta}X^{-\alpha}
Y^{-1-j}L_h(x)^{j+1},\quad\alpha\in\mathbb N_0.
\end{align}
Finally, $a_{j+1}=-(i/2)a_0I_j$. By Leibniz' rule,
\eqref{eq:rho-bound} and \eqref{eq:Ij-bound},
\[
|\partial_x^\alpha\partial_h^\beta a_{j+1}(x,h)|
\leq C_{\alpha\beta j}h^{-\beta}X^{-1/2-\alpha}
Y^{-3/2-j}L_h(x)^{j+1}.
\]
This is the assertion $a_{j+1}\in\mathcal S^{j+1}_{-1/2-(j+1)}$.
\end{proof}

\medskip

{\small
\subsubsection*{Acknowledgments} The authors would like to thank Ludwik D{\ogonekaccent a}browski for drawing their attention to the problem.  K.~Taira and M.~Wrochna gratefully acknowledge support from the Research Institute for Mathematical Sciences in Kyoto (RIMS). K.~Taira was partially supported by JSPS KAKENHI Grant Number JP23K13004. M.~Wrochna was supported by the Dutch Research Council (NWO) through the grant OCENW.M.24.071, and by the Erwin Schr\"odinger Institute in Vienna (ESI) through the program ``Spectral Theory and Mathematical Relativity'', 2023. \medskip }

\end{document}